\documentclass[10pt,reqno]{amsart}
\usepackage{amsmath,amssymb}
\usepackage[colorlinks=true,linkcolor=magenta,citecolor=blue]{hyperref}
\usepackage[margin=1in]{geometry}
\usepackage{enumitem} 
\usepackage{mathabx}
\usepackage{epsfig}

\newtheorem{dn}{Definition}[section]
\newtheorem{bdt}{Inequality}[section]
\newtheorem{dl}{Theorem}[section]
\newtheorem{md}{Proposition}[section]
\newtheorem{bd}{Lemma}[section]
\newtheorem{hq}{Corollary}[section]
\newtheorem{nx}{Remark}[section]
\newtheorem{vd}{Example}[section]
\newcommand{\R}{\mathbb{R}}

\newcommand{\Z}{\mathbb{Z}}
\newcommand{\N}{\mathbb{N}}
\newcommand{\e}{\varepsilon}
\newcommand{\ity}{\infty}
\newcommand{\f}{\frac}

\newcommand{\bbd}{\begin{bd}}
\newcommand{\ebd}{\end{bd}}
\newcommand{\bbdt}{\begin{bdt}}
\newcommand{\ebdt}{\end{bdt}}
\newcommand{\bdn}{\begin{dn}}
\newcommand{\edn}{\end{dn}}
\newcommand{\bhq}{\begin{hq}}
\newcommand{\ehq}{\end{hq}}
\newcommand{\bdl}{\begin{dl}}
\newcommand{\edl}{\end{dl}}
\newcommand{\bnx}{\begin{nx}}
\newcommand{\enx}{\end{nx}}
\newcommand{\bmd}{\begin{md}}
\newcommand{\emd}{\end{md}}
\newcommand{\bvd}{\begin{vd}}
\newcommand{\evd}{\end{vd}}

\title[Weakly coupled systems of semi-linear structurally damped $\sigma$-evolution models]{Global existence for weakly coupled systems of semi-linear structurally damped $\sigma$-evolution models with different power nonlinearities}
\author{Tuan Anh Dao}
\address{Tuan Anh Dao \hfill\break
$\quad$ School of Applied Mathematics and Informatics, Hanoi University of Science and Technology, No.1 Dai Co Viet road, Hanoi, Vietnam \hfill\break
Faculty for Mathematics and Computer Science, TU Bergakademie Freiberg, Pr\"{u}ferstr. 9, 09596, Freiberg, Germany}
\email{daotuananh.fami@gmail.com}

\begin{document}
	
\begin{abstract}
In this paper, we study the Cauchy problems for weakly coupled systems of semi-linear structurally damped $\sigma$-evolution models with different power nonlinearities. By assuming additional $L^m$ regularity on the initial data, with $m \in [1,2)$, we use $(L^m \cap L^2)- L^2$ and $L^2- L^2$ estimates for solutions to the corresponding linear Cauchy problems to prove the global (in time) existence of small data Sobolev solutions to the weakly coupled systems of semi-linear models from suitable function spaces.
\end{abstract}

	\maketitle
	
	\tableofcontents
	
\section{Introduction and main results} \label{Sec.main}
There are several recent papers (see, for example, \cite{DabbiccoEbert,DuongKainaneReissig}) concerning the proof of global (in time) existence for semi-linear structurally damped $\sigma$-evolution equations. In particular, the authors studied the two Cauchy problems:
$$ u_{tt}+ (-\Delta)^\sigma u+ (-\Delta)^{\delta} u_t= |u|^p ,\,\,\, u(0,x)= u_0(x),\,\,\, u_t(0,x)=u_1(x), $$
and
$$ u_{tt}+ (-\Delta)^\sigma u+ (-\Delta)^{\delta} u_t= |u_t|^p ,\,\,\, u(0,x)= u_0(x),\,\,\, u_t(0,x)=u_1(x) $$
with $\sigma \ge 1$ and $\delta \in (0,\frac{\sigma}{2}]$. Here the use of $(L^1 \cap L^2)- L^2$ estimates to the corresponding linear Cauchy problems, i.e., the mixing of additional $L^1$ regularity for the data on the basis of $L^2- L^2$ estimates came into play to investigate these semi-linear equations in \cite{DuongKainaneReissig}. In this paper, we consider the following three Cauchy problems for weakly coupled systems of semi-linear structurally damped $\sigma$-evolution equations:
\begin{equation}
\begin{cases}
u_{tt}+ (-\Delta)^\sigma u+ (-\Delta)^{\delta} u_t=|v|^p,\,\,\,  v_{tt}+ (-\Delta)^\sigma v+ (-\Delta)^{\delta} v_t=|u|^q, \\
u(0,x)= u_0(x),\,\, u_t(0,x)=u_1(x),\,\, v(0,x)= v_0(x),\,\, v_t(0,x)=v_1(x), \label{pt1.1}
\end{cases}
\end{equation}
and
\begin{equation}
\begin{cases}
u_{tt}+ (-\Delta)^\sigma u+ (-\Delta)^{\delta} u_t=|v_t|^p,\,\,\,  v_{tt}+ (-\Delta)^\sigma v+ (-\Delta)^{\delta} v_t=|u_t|^q, \\
u(0,x)= u_0(x),\,\, u_t(0,x)=u_1(x),\,\, v(0,x)= v_0(x),\,\, v_t(0,x)=v_1(x), \label{pt1.2}
\end{cases}
\end{equation}
and
\begin{equation}
\begin{cases}
u_{tt}+ (-\Delta)^\sigma u+ (-\Delta)^{\delta} u_t=|v|^p,\,\,\,  v_{tt}+ (-\Delta)^\sigma v+ (-\Delta)^{\delta} v_t=|u_t|^q, \\
u(0,x)= u_0(x),\,\, u_t(0,x)=u_1(x),\,\, v(0,x)= v_0(x),\,\, v_t(0,x)=v_1(x), \label{pt1.4}
\end{cases}
\end{equation}
with $\sigma \ge 1$, $\delta \in (0,\sigma)$ and $p,\, q >1$. The corresponding linear models with vanishing right-hand side are
\begin{equation}
w_{tt}+ (-\Delta)^\sigma w+ (-\Delta)^{\delta} w_t=0,\,\, w(0,x)= w_0(x),\,\, w_t(0,x)= w_1(x). \label{pt1.3}
\end{equation}

The first motivation of the present paper is to get sharp $(L^m \cap L^2)- L^2$ estimates with $m \in [1,2)$ for the solutions to (\ref{pt1.3}). Having these estimates play a fundamental role in the treatment of corresponding semi-linear models. For this reason, the second motivation is prove the global (in time) existence of small data Sobolev solutions to (\ref{pt1.1}), (\ref{pt1.2}) and (\ref{pt1.4}) by applying the obtained linear estimates and some developed tools from Harmonic Analysis (see, for instance, \cite{DabbiccoReissig, Palmierithesis}). Finally, we also prove the optimality of our exponents when $\sigma$ and $\delta$ are integers.\medskip

$\qquad$ \textbf{Notations}\medskip

Throughout the present paper, we use the following notations.
\begin{itemize}[leftmargin=*]
\item We write $f\lesssim g$ when there exists a constant $C>0$ such that $f\le Cg$, and $f \approx g$ when $g\lesssim f\lesssim g$.
\item As usual, $H^{a}$ and $\dot{H}^{a}$, with $a \ge 0$, denote Bessel and Riesz potential spaces based on $L^2$. Here $\big<D\big>^{a}$ and $|D|^{a}$ stand for the pseudo-differential operators with symbols $\big<\xi\big>^{a}$ and $|\xi|^{a}$, respectively.
\item We denote $[s]^+:= \max\{s,0\}$ as the positive part of $s \in \R$, and $\lceil s \rceil:= \min \big\{k \in \Z \,\, : \,\, k\ge s \big\}$.
\item We fix the constants $\mathtt{k}^-:= \min\{\sigma,\, 2\delta\}$ and $\mathtt{k}^+:= \max\{\sigma,\, 2\delta\}$. Moreover, we fix the constant $m_0=\frac{2m}{2-m}$, that is, $\frac{1}{m_0}=\frac{1}{m}- \frac{1}{2}$ with $m \in [1,2)$.
\item Finally, we introduce the spaces $\mathcal{A}^{s}_{m}:= \big(L^m \cap H^{s}\big) \times \big(L^m \cap H^{[s-\mathtt{k}^+]^+}\big)$ with the norm
$$\|(u_0,u_1)\|_{\mathcal{A}^{s}_{m}}:=\|u_0\|_{L^m}+ \|u_0\|_{H^{s}}+ \|u_1\|_{L^m}+ \|u_1\|_{H^{[s-\mathtt{k}^+]^+}}, $$
where $s \ge 0$ and $m \in [1,2)$.
\end{itemize}

$\qquad$ \textbf{Main results}\medskip

Let us state the main results that will be proved in this paper.\medskip

In the first case, we obtain solutions to (\ref{pt1.1}) from energy space.\medskip

\noindent\textbf{Theorem 1-A.}\textit{ Let $m \in [1,2)$ and $n> m_0 \mathtt{k}^-$. We assume the conditions
\begin{align}
&\frac{2}{m} \le p,\,q  < \ity &\qquad \text{ if }&\, n \le 2\mathtt{k}^+, \label{GN1A1} \\
&\frac{2}{m} \le p,\,q \le \frac{n}{n- 2\mathtt{k}^+} &\qquad \text{ if }&\, n \in \Big(2\mathtt{k}^+, \frac{4\mathtt{k}^+}{2-m}\Big]. \label{GN1A2}
\end{align}
Moreover, we suppose the following conditions:
\begin{equation} \label{exponent1A}
m \Big(\mathtt{k}^- + \frac{(\mathtt{k}^+ +\sigma)(1+\max\{p,\,q\})}{pq-1}\Big) < n, \text{ and }\min\{p,\,q\} \le 1+\frac{m(\mathtt{k}^+ +\sigma)}{n- m\mathtt{k}^-} < \max\{p,\,q\}.
\end{equation}
Then, there exists a constant $\e>0$ such that for any small data
\begin{equation*}
\big((u_0,u_1),\, (v_0,v_1) \big) \in \mathcal{A}^{\mathtt{k}^+}_{m} \times \mathcal{A}^{\mathtt{k}^+}_{m} \text{ satisfying the assumption } \|(u_0,u_1)\|_{\mathcal{A}^{\mathtt{k}^+}_{m}}+ \|(v_0,v_1)\|_{\mathcal{A}^{\mathtt{k}^+}_{m}} \le \e,
\end{equation*}
we have a uniquely determined global (in time) small data energy solution
$$ (u,v) \in \Big(C([0,\ity),H^{\mathtt{k}^+})\cap C^1([0,\ity),L^2)\Big)^2 $$
to (\ref{pt1.1}). The following estimates hold:
\begin{align}
\|u(t,\cdot)\|_{L^2}& \lesssim (1+t)^{-\frac{n}{2(\mathtt{k}^+ -\delta)}(\frac{1}{m}-\frac{1}{2})+ \frac{\mathtt{k}^-}{2(\mathtt{k}^+ -\delta)}+[\e(p)]^+} \big(\|(u_0,u_1)\|_{\mathcal{A}^{\mathtt{k}^+}_{m}}+ \|(v_0,v_1)\|_{\mathcal{A}^{\mathtt{k}^+}_{m}}\big), \label{decayrate1A1} \\
\big\||D|^{\mathtt{k}^+} u(t,\cdot)\big\|_{L^2}& \lesssim (1+t)^{-\frac{n}{2(\mathtt{k}^+ -\delta)}(\frac{1}{m}-\frac{1}{2})- \frac{\mathtt{k}^+- \mathtt{k}^-}{2(\mathtt{k}^+ -\delta)}+[\e(p)]^+} \big(\|(u_0,u_1)\|_{\mathcal{A}^{\mathtt{k}^+}_{m}}+ \|(v_0,v_1)\|_{\mathcal{A}^{\mathtt{k}^+}_{m}}\big), \label{decayrate1A2} \\
\|u_t(t,\cdot)\|_{L^2}& \lesssim (1+t)^{-\frac{n}{2(\mathtt{k}^+ -\delta)}(\frac{1}{m}-\frac{1}{2})- \frac{\sigma- \mathtt{k}^-}{\mathtt{k}^+ -\delta}+[\e(p)]^+} \big(\|(u_0,u_1)\|_{\mathcal{A}^{\mathtt{k}^+}_{m}}+ \|(v_0,v_1)\|_{\mathcal{A}^{\mathtt{k}^+}_{m}}\big), \label{decayrate1A3} \\
\|v(t,\cdot)\|_{L^2}& \lesssim (1+t)^{-\frac{n}{2(\mathtt{k}^+ -\delta)}(\frac{1}{m}-\frac{1}{2})+ \frac{\mathtt{k}^-}{2(\mathtt{k}^+ -\delta)}+[\e(q)]^+} \big(\|(u_0,u_1)\|_{\mathcal{A}^{\mathtt{k}^+}_{m}}+ \|(v_0,v_1)\|_{\mathcal{A}^{\mathtt{k}^+}_{m}}\big), \label{decayrate1A4} \\
\big\||D|^{\mathtt{k}^+} v(t,\cdot)\big\|_{L^2}& \lesssim (1+t)^{-\frac{n}{2(\mathtt{k}^+ -\delta)}(\frac{1}{m}-\frac{1}{2})- \frac{\mathtt{k}^+- \mathtt{k}^-}{2(\mathtt{k}^+ -\delta)}+[\e(q)]^+} \big(\|(u_0,u_1)\|_{\mathcal{A}^{\mathtt{k}^+}_{m}}+ \|(v_0,v_1)\|_{\mathcal{A}^{\mathtt{k}^+}_{m}}\big), \label{decayrate1A5} \\
\|v_t(t,\cdot)\|_{L^2}& \lesssim (1+t)^{-\frac{n}{2(\mathtt{k}^+ -\delta)}(\frac{1}{m}-\frac{1}{2})- \frac{\sigma- \mathtt{k}^-}{\mathtt{k}^+ -\delta}+[\e(q)]^+} \big(\|(u_0,u_1)\|_{\mathcal{A}^{\mathtt{k}^+}_{m}}+ \|(v_0,v_1)\|_{\mathcal{A}^{\mathtt{k}^+}_{m}}\big), \label{decayrate1A6}
\end{align}
where $\e(p):= 1-\frac{n}{2m(\mathtt{k}^+ -\delta)}(p-1)+ \frac{p\mathtt{k}^-}{2(\mathtt{k}^+ -\delta)}+\e$ and $\e(q):= 1-\frac{n}{2m(\mathtt{k}^+ -\delta)}(q-1)+ \frac{q\mathtt{k}^-}{2(\mathtt{k}^+ -\delta)}+\e$ with a sufficiently small positive $\e$.}\medskip

\noindent\textbf{Theorem 1-B.}\textit{ Under the assumptions of Theorem 1-A, if condition (\ref{exponent1A}) is replaced by
\begin{equation} \label{exponent1B}
\min\{p,\,q\}> 1+\frac{m(\mathtt{k}^+ +\sigma)}{n- m\mathtt{k}^-},
\end{equation}
then we have the same conclusions of Theorem 1-A. But the estimates (\ref{decayrate1A1})-(\ref{decayrate1A6}) are modified in the following way:
\begin{align*}
\|(u,v)(t,\cdot)\|_{L^2}& \lesssim (1+t)^{-\frac{n}{2(\mathtt{k}^+ -\delta)}(\frac{1}{m}-\frac{1}{2})+ \frac{\mathtt{k}^-}{2(\mathtt{k}^+ -\delta)}} \big(\|(u_0,u_1)\|_{\mathcal{A}^{\mathtt{k}^+}_{m}}+ \|(v_0,v_1)\|_{\mathcal{A}^{\mathtt{k}^+}_{m}}\big), \\
\big\|\big(|D|^{\mathtt{k}^+} u, |D|^{\mathtt{k}^+} v\big)(t,\cdot)\big\|_{L^2}& \lesssim (1+t)^{-\frac{n}{2(\mathtt{k}^+ -\delta)}(\frac{1}{m}-\frac{1}{2})- \frac{\mathtt{k}^+- \mathtt{k}^-}{2(\mathtt{k}^+ -\delta)}} \big(\|(u_0,u_1)\|_{\mathcal{A}^{\mathtt{k}^+}_{m}}+ \|(v_0,v_1)\|_{\mathcal{A}^{\mathtt{k}^+}_{m}}\big), \\
\|(u_t,v_t)(t,\cdot)\|_{L^2}& \lesssim (1+t)^{-\frac{n}{2(\mathtt{k}^+ -\delta)}(\frac{1}{m}-\frac{1}{2})- \frac{\sigma- \mathtt{k}^-}{\mathtt{k}^+ -\delta}} \big(\|(u_0,u_1)\|_{\mathcal{A}^{\mathtt{k}^+}_{m}}+ \|(v_0,v_1)\|_{\mathcal{A}^{\mathtt{k}^+}_{m}}\big).
\end{align*}}

In the second case, we obtain Sobolev solutions to (\ref{pt1.1}).\medskip

\noindent\textbf{Theorem 2-A.}\textit{ Let $0< s_1 \le s_2 < \mathtt{k}^+$, $m \in [1,2)$ and $n> m_0 \mathtt{k}^-$. We assume the conditions
\begin{align}
&\frac{2}{m} \le p,\, q < \ity & & &\quad \text{ if }&\, n \le 2s_1, \label{GN2A1} \\
&\frac{2}{m} \le p < \ity, &\quad &\frac{2}{m} \le q  \le \frac{n}{n- 2s_1} &\quad \text{ if }&\, 2s_1 < n \le \min\Big\{2s_2,\, \frac{4s_1}{2-m}\Big\}, \label{GN2A2} \\
&\frac{2}{m} \le p \le \frac{n}{n- 2s_2}, &\quad &\frac{2}{m} \le q  \le \frac{n}{n- 2s_1} &\quad \text{ if }&\, 2s_2 < n \le \frac{4s_1}{2-m}. \label{GN2A3}
\end{align}
Moreover, we suppose the following conditions:
\begin{equation} \label{exponent2A}
m \Big(\mathtt{k}^- + \frac{(\mathtt{k}^+ +\sigma)(1+\max\{p,\,q\})}{pq-1}\Big) < n, \text{ and }\min\{p,\,q\} \le 1+\frac{m(\mathtt{k}^+ +\sigma)}{n- m\mathtt{k}^-} < \max\{p,\,q\}.
\end{equation}
Then, there exists a constant $\e>0$ such that for any small data
\begin{equation*}
\big((u_0,u_1),\, (v_0,v_1) \big) \in \mathcal{A}^{s_1}_{m} \times \mathcal{A}^{s_2}_{m} \text{ satisfying the assumption } \|(u_0,u_1)\|_{\mathcal{A}^{s_1}_{m}}+ \|(v_0,v_1)\|_{\mathcal{A}^{s_2}_{m}} \le \e,
\end{equation*}
we have a uniquely determined global (in time) small data energy solution
$$ (u,v) \in C([0,\ity),H^{s_1}) \times C([0,\ity),H^{s_2}) $$
to (\ref{pt1.1}). The following estimates hold:
\begin{align}
\|u(t,\cdot)\|_{L^2}& \lesssim (1+t)^{-\frac{n}{2(\mathtt{k}^+ -\delta)}(\frac{1}{m}-\frac{1}{2})+ \frac{\mathtt{k}^-}{2(\mathtt{k}^+ -\delta)}+[\e(p)]^+} \big(\|(u_0,u_1)\|_{\mathcal{A}^{s_1}_{m}}+ \|(v_0,v_1)\|_{\mathcal{A}^{s_2}_{m}}\big), \label{decayrate2A1} \\
\big\||D|^{s_1} u(t,\cdot)\big\|_{L^2}& \lesssim (1+t)^{-\frac{n}{2(\mathtt{k}^+ -\delta)}(\frac{1}{m}-\frac{1}{2})- \frac{s_1- \mathtt{k}^-}{2(\mathtt{k}^+ -\delta)}+[\e(p)]^+} \big(\|(u_0,u_1)\|_{\mathcal{A}^{s_1}_{m}}+ \|(v_0,v_1)\|_{\mathcal{A}^{s_2}_{m}}\big), \label{decayrate2A2} \\
\|v(t,\cdot)\|_{L^2}& \lesssim (1+t)^{-\frac{n}{2(\mathtt{k}^+ -\delta)}(\frac{1}{m}-\frac{1}{2})+ \frac{\mathtt{k}^-}{2(\mathtt{k}^+ -\delta)}+[\e(q)]^+} \big(\|(u_0,u_1)\|_{\mathcal{A}^{s_1}_{m}}+ \|(v_0,v_1)\|_{\mathcal{A}^{s_2}_{m}}\big), \label{decayrate2A3} \\
\big\||D|^{s_2} v(t,\cdot)\big\|_{L^2}& \lesssim (1+t)^{-\frac{n}{2(\mathtt{k}^+ -\delta)}(\frac{1}{m}-\frac{1}{2})- \frac{s_2- \mathtt{k}^-}{2(\mathtt{k}^+ -\delta)}+[\e(q)]^+} \big(\|(u_0,u_1)\|_{\mathcal{A}^{s_1}_{m}}+ \|(v_0,v_1)\|_{\mathcal{A}^{s_2}_{m}}\big), \label{decayrate2A4}
\end{align}
where $\e(p):= 1-\frac{n}{2m(\mathtt{k}^+ -\delta)}(p-1)+ \frac{p\mathtt{k}^-}{2(\mathtt{k}^+ -\delta)}+\e$ and $\e(q):= 1-\frac{n}{2m(\mathtt{k}^+ -\delta)}(q-1)+ \frac{q\mathtt{k}^-}{2(\mathtt{k}^+ -\delta)}+\e$ with a sufficiently small positive $\e$.}\medskip

\noindent\textbf{Theorem 2-B.}\textit{ Under the assumptions of Theorem 2-A, if condition (\ref{exponent2A}) is replaced by
\begin{equation} \label{exponent2B}
\min\{p,\,q\}> 1+\frac{m(\mathtt{k}^+ +\sigma)}{n- m\mathtt{k}^-},
\end{equation}
then we have the same conclusions of Theorem 2-A. But the estimates (\ref{decayrate2A1})-(\ref{decayrate2A4}) are modified in the following way:
\begin{align*}
\|(u,v)(t,\cdot)\|_{L^2}& \lesssim (1+t)^{-\frac{n}{2(\mathtt{k}^+ -\delta)}(\frac{1}{m}-\frac{1}{2})+ \frac{\mathtt{k}^-}{2(\mathtt{k}^+ -\delta)}} \big(\|(u_0,u_1)\|_{\mathcal{A}^{s_1}_{m}}+ \|(v_0,v_1)\|_{\mathcal{A}^{s_2}_{m}}\big), \\
\big\||D|^{s_1} u(t,\cdot)\big\|_{L^2}& \lesssim (1+t)^{-\frac{n}{2(\mathtt{k}^+ -\delta)}(\frac{1}{m}-\frac{1}{2})- \frac{s_1- \mathtt{k}^-}{2(\mathtt{k}^+ -\delta)}} \big(\|(u_0,u_1)\|_{\mathcal{A}^{s_1}_{m}}+ \|(v_0,v_1)\|_{\mathcal{A}^{s_2}_{m}}\big), \\
\big\||D|^{s_2} v(t,\cdot)\big\|_{L^2}& \lesssim (1+t)^{-\frac{n}{2(\mathtt{k}^+ -\delta)}(\frac{1}{m}-\frac{1}{2})- \frac{s_2- \mathtt{k}^-}{2(\mathtt{k}^+ -\delta)}} \big(\|(u_0,u_1)\|_{\mathcal{A}^{s_1}_{m}}+ \|(v_0,v_1)\|_{\mathcal{A}^{s_2}_{m}}\big).
\end{align*}}

\begin{nx}
\fontshape{n}
\selectfont
Due to the first condition in (\ref{exponent1A}) and (\ref{exponent2A}), we may imply that at least one among $[\e(p)]^+$ or $[\e(q)]^+$ in Theorems 1-A and 2-A is equal to zero.
\end{nx}

\begin{nx}
\fontshape{n}
\selectfont
Let us compare our results between Theorems A and B. First, we can see that the decay rates for the solutions to (\ref{pt1.1}) in Theorems 1-A and 2-A are worse than those for solutions to the corresponding linear models, that is, we allow some loss of decay (see more \cite{DabbiccoEbert2014,DaoReissig}). This phenomenon is related to some of the used techniques in our proofs. Moreover, in Theorems 1-B and 2-B there appear the same decay rates as in the estimates for the solutions to (\ref{pt1.3}), i.e., no loss of decay appears. Here we want to underline that allowing loss of decay brings some benifits to relax the restrictions to the admissible exponents $p$ and $q$ in comparison (\ref{exponent1A}) with (\ref{exponent1B}) (respectively (\ref{exponent2A}) with (\ref{exponent2B})). In particular, in (\ref{exponent1A}) and (\ref{exponent2A}) we allow one exponent $p$ or $q$ below the exponent $1+\frac{m(\mathtt{k}^+ +\sigma)}{n- m\mathtt{k}^-}$, whereas we need to guarantee both exponents $p$ and $q$ above the exponent $1+\frac{m(\mathtt{k}^+ +\sigma)}{n- m\mathtt{k}^-}$ in (\ref{exponent1B}) and (\ref{exponent2B}). However, we pay with further conditions for space dimension $n$ as in (\ref{exponent1A}) and (\ref{exponent2A}) . 
\end{nx}

\begin{nx}
\fontshape{n}
\selectfont
To the semi-linear models (\ref{pt1.1}), by setting formally $\sigma=1$, $\delta=0$ and $m=1$ we observe that the admissible exponents $p$ and $q$ in Theorem 1-A are consistent with those in \cite{SunWang} for space dimensions $n=1,3$.
\end{nx}

In the third case, we obtain solutions to (\ref{pt1.1}) belonging to the energy space with a suitable higher regularity.\medskip

\noindent\textbf{Theorem 3.}\textit{ Let $\mathtt{k}^+ < s_1 \le s_2 \le \frac{n}{2}+\mathtt{k}^+$, $s_2-s_1< \mathtt{k}^+$, $m \in [1,2)$ and $n> m_0 \mathtt{k}^-$. We assume the conditions
\begin{align}
&1+ \lceil s_1- \mathtt{k}^+ \rceil< p < \ity, &\quad &1+ \lceil s_2- \mathtt{k}^+ \rceil< q < \ity &\text{ if }&\, n \le 2s_1, \label{GN3A1} \\
&1+ \lceil s_1- \mathtt{k}^+ \rceil< p < \ity, &\quad &1+ \lceil s_2- \mathtt{k}^+ \rceil< q \le 1+\frac{2\mathtt{k}^+}{n- 2s_1} &\text{ if }&\, 2s_1 < n \le 2s_2, \label{GN3A2} \\
&1+ \lceil s_1- \mathtt{k}^+ \rceil< p \le 1+\frac{2\mathtt{k}^+}{n- 2s_2}, &\quad &1+ \lceil s_2- \mathtt{k}^+ \rceil< q \le 1+\frac{2\mathtt{k}^+}{n- 2s_1} &\text{ if }&\, n > 2s_2. \label{GN3A3}
\end{align}
Moreover, we suppose $\min\{p,\,q\} > 1+ \frac{m(\mathtt{k}^+ +\sigma)}{n- m\mathtt{k}^-}$ satisfying the following conditions:
\begin{equation} \label{exponent3A}
p \ge 1+ \frac{ms_1}{n- m\mathtt{k}^-} \text{ and } q \ge 1+ \frac{ms_2}{n- m\mathtt{k}^-}.
\end{equation}
Then, there exists a constant $\e>0$ such that for any small data
\begin{equation*}
\big((u_0,u_1),\, (v_0,v_1) \big) \in \mathcal{A}^{s_1}_{m} \times \mathcal{A}^{s_2}_{m} \text{ satisfying the assumption } \|(u_0,u_1)\|_{\mathcal{A}^{s_1}_{m}}+ \|(v_0,v_1)\|_{\mathcal{A}^{s_2}_{m}} \le \e,
\end{equation*}
we have a uniquely determined global (in time) small data energy solution
$$ (u,v) \in \Big(C([0,\ity),H^{s_1})\cap C^1([0,\ity),H^{s_1- \mathtt{k}^+})\Big) \times \Big(C([0,\ity),H^{s_2})\cap C^1([0,\ity),H^{s_2- \mathtt{k}^+})\Big) $$
to (\ref{pt1.1}). The following estimates hold:
\begin{align}
\|(u,v)(t,\cdot)\|_{L^2}& \lesssim (1+t)^{-\frac{n}{2(\mathtt{k}^+ -\delta)}(\frac{1}{m}-\frac{1}{2})+ \frac{\mathtt{k}^-}{2(\mathtt{k}^+ -\delta)}} \big(\|(u_0,u_1)\|_{\mathcal{A}^{s_1}_{m}}+ \|(v_0,v_1)\|_{\mathcal{A}^{s_2}_{m}}\big), \label{decayrate3A1} \\
\|(u_t,v_t)(t,\cdot)\|_{L^2}& \lesssim (1+t)^{-\frac{n}{2(\mathtt{k}^+ -\delta)}(\frac{1}{m}-\frac{1}{2})- \frac{\sigma- \mathtt{k}^-}{\mathtt{k}^+ -\delta}} \big(\|(u_0,u_1)\|_{\mathcal{A}^{s_1}_{m}}+ \|(v_0,v_1)\|_{\mathcal{A}^{s_2}_{m}}\big), \label{decayrate3A2} \\
\big\||D|^{s_1} u(t,\cdot)\big\|_{L^2}& \lesssim (1+t)^{-\frac{n}{2(\mathtt{k}^+ -\delta)}(\frac{1}{m}-\frac{1}{2})- \frac{s_1- \mathtt{k}^-}{2(\mathtt{k}^+ -\delta)}} \big(\|(u_0,u_1)\|_{\mathcal{A}^{s_1}_{m}}+ \|(v_0,v_1)\|_{\mathcal{A}^{s_2}_{m}}\big), \label{decayrate3A3} \\
\||D|^{s_1- \mathtt{k}^+}u_t(t,\cdot)\|_{L^2}& \lesssim (1+t)^{-\frac{n}{2(\mathtt{k}^+ -\delta)}(\frac{1}{m}-\frac{1}{2})- \frac{s_1+ \mathtt{k}^+ - 4\delta}{2(\mathtt{k}^+ -\delta)}} \big(\|(u_0,u_1)\|_{\mathcal{A}^{s_1}_{m}}+ \|(v_0,v_1)\|_{\mathcal{A}^{s_2}_{m}}\big), \label{decayrate3A4} \\
\big\||D|^{s_2} v(t,\cdot)\big\|_{L^2}& \lesssim (1+t)^{-\frac{n}{2(\mathtt{k}^+ -\delta)}(\frac{1}{m}-\frac{1}{2})- \frac{s_2- \mathtt{k}^-}{2(\mathtt{k}^+ -\delta)}} \big(\|(u_0,u_1)\|_{\mathcal{A}^{s_1}_{m}}+ \|(v_0,v_1)\|_{\mathcal{A}^{s_2}_{m}}\big), \label{decayrate3A5} \\
\||D|^{s_2- \mathtt{k}^+}v_t(t,\cdot)\|_{L^2}& \lesssim (1+t)^{-\frac{n}{2(\mathtt{k}^+ -\delta)}(\frac{1}{m}-\frac{1}{2})- \frac{s_2+ \mathtt{k}^+ - 4\delta}{2(\mathtt{k}^+ -\delta)}} \big(\|(u_0,u_1)\|_{\mathcal{A}^{s_1}_{m}}+ \|(v_0,v_1)\|_{\mathcal{A}^{s_2}_{m}}\big). \label{decayrate3A6}
\end{align}}

\begin{nx} \label{nx1.3}
\fontshape{n}
\selectfont
If we assume that $\mathtt{k}^+ < s_1 \le \mathtt{k}^+ +\sigma$ in Theorem $3$, then it is also reasonable to allow a loss of decay. Indeed, in the same treament as we did in Theorem 1-A we may replace the condition $\min\{p,\,q\} > 1+ \frac{m(\mathtt{k}^+ +\sigma)}{n- m\mathtt{k}^-}$ in Theorem $3$ by (\ref{exponent1A}) in Theorem 1-A. Moreover, if we consider $s_1> \mathtt{k}^+ +\sigma$ in Theorem $4$, then the condition $\min\{p,\,q\} > 1+ \frac{m(\mathtt{k}^+ +\sigma)}{n- m\mathtt{k}^-}$ is contained in (\ref{exponent3A}).
\end{nx}

In the fourth case, we obtain large regular solutions to (\ref{pt1.1}) by using the fractional Sobolev embedding.\medskip

\noindent\textbf{Theorem 4.}\textit{ Let $s_2 \ge s_1 > \frac{n}{2}+\mathtt{k}^+$, $s_2-s_1 \le \mathtt{k}^+$, $m \in [1,2)$ and $n> m_0 \mathtt{k}^-$. We assume the conditions
\begin{equation} \label{exponent4A1}
p > 1+\max\{s_1- \mathtt{k}^+,\,1\} \text{ and } q > 1+\max\{s_2- \mathtt{k}^+,\,1\}.
\end{equation}
Moreover, we suppose $\min\{p,\,q\} > 1+ \frac{m(\mathtt{k}^+ +\sigma)}{n- m\mathtt{k}^-}$ satisfying the following conditions:
\begin{equation} \label{exponent4A2}
p \ge 1+ \frac{ms_1}{n- m\mathtt{k}^-} \text{ and } q \ge 1+ \frac{ms_2}{n- m\mathtt{k}^-}.
\end{equation}
Then, there exists a constant $\e>0$ such that for any small data
\begin{equation*}
\big((u_0,u_1),\, (v_0,v_1) \big) \in \mathcal{A}^{s_1}_{m} \times \mathcal{A}^{s_2}_{m} \text{ satisfying the assumption } \|(u_0,u_1)\|_{\mathcal{A}^{s_1}_{m}}+ \|(v_0,v_1)\|_{\mathcal{A}^{s_2}_{m}} \le \e,
\end{equation*}
we have a uniquely determined global (in time) small data energy solution
$$ (u,v) \in \Big(C([0,\ity),H^{s_1})\cap C^1([0,\ity),H^{s_1- \mathtt{k}^+})\Big) \times \Big(C([0,\ity),H^{s_2})\cap C^1([0,\ity),H^{s_2- \mathtt{k}^+})\Big) $$
to (\ref{pt1.1}).  Moreover, the estimates (\ref{decayrate3A1})-(\ref{decayrate3A6}) hold.}\medskip

\begin{nx}
\fontshape{n}
\selectfont
Like in Remark \ref{nx1.3}, if we assume that $\frac{n}{2}+\mathtt{k}^+ < s_1 \le \mathtt{k}^+ +\sigma$ in Theorem $4$, then allowing a loss of decay implies that the condition $\min\{p,\,q\} > 1+ \frac{m(\mathtt{k}^+ +\sigma)}{n- m\mathtt{k}^-}$ in Theorem $4$ is replaced by (\ref{exponent1A}) in Theorem 1-A. Moreover, if we consider $s_1> \mathtt{k}^+ +\sigma$ in Theorem $4$, then the condition $\min\{p,\,q\} > 1+ \frac{m(\mathtt{k}^+ +\sigma)}{n- m\mathtt{k}^-}$ is contained in (\ref{exponent4A2}).
\end{nx}

In the fifth case, we obtain large regular solutions to (\ref{pt1.2}) by using the fractional Sobolev embedding.\medskip

\noindent\textbf{Theorem 5.}\textit{ Let $s > \frac{n}{2}+\mathtt{k}^+$, $m \in [1,2)$ and $n> m_0 \mathtt{k}^-$. We assume the following conditions:
\begin{equation} \label{exponent5A1}
\min\{p,\,q\} > 1+\max\Big\{\frac{2m\delta}{n},\, s- \mathtt{k}^+,\,1\Big\},
\end{equation}
and
\begin{equation} \label{exponent5A2}
\min\{p,\,q\} \ge 1+ \frac{m(s+ \mathtt{k}^- -2\sigma)}{n+ 2m(\mathtt{k}^+ -2\delta)}.
\end{equation}
Then, there exists a constant $\e>0$ such that for any small data
\begin{equation*}
\big((u_0,u_1),\, (v_0,v_1) \big) \in \mathcal{A}^{s}_{m} \times \mathcal{A}^{s}_{m} \text{ satisfying the assumption } \|(u_0,u_1)\|_{\mathcal{A}^{s}_{m}}+ \|(v_0,v_1)\|_{\mathcal{A}^{s}_{m}} \le \e,
\end{equation*}
we have a uniquely determined global (in time) small data energy solution
$$ (u,v) \in \Big(C([0,\ity),H^s)\cap C^1([0,\ity),H^{s- \mathtt{k}^+})\Big)^2 $$
to (\ref{pt1.1}). The following estimates hold:
\begin{align}
\|(u,v)(t,\cdot)\|_{L^2}& \lesssim (1+t)^{-\frac{n}{2(\mathtt{k}^+ -\delta)}(\frac{1}{m}-\frac{1}{2})+ \frac{\mathtt{k}^-}{2(\mathtt{k}^+ -\delta)}} \big(\|(u_0,u_1)\|_{\mathcal{A}^s_{m}}+ \|(v_0,v_1)\|_{\mathcal{A}^s_{m}}\big), \label{decayrate5A1} \\
\|(u_t,v_t)(t,\cdot)\|_{L^2}& \lesssim (1+t)^{-\frac{n}{2(\mathtt{k}^+ -\delta)}(\frac{1}{m}-\frac{1}{2})- \frac{\sigma- \mathtt{k}^-}{\mathtt{k}^+ -\delta}} \big(\|(u_0,u_1)\|_{\mathcal{A}^s_{m}}+ \|(v_0,v_1)\|_{\mathcal{A}^s_{m}}\big), \label{decayrate5A2} \\
\big\|\big(|D|^{s} u, |D|^{s} v \big)(t,\cdot)\big\|_{L^2}& \lesssim (1+t)^{-\frac{n}{2(\mathtt{k}^+ -\delta)}(\frac{1}{m}-\frac{1}{2})- \frac{s- \mathtt{k}^-}{2(\mathtt{k}^+ -\delta)}} \big(\|(u_0,u_1)\|_{\mathcal{A}^s_{m}}+ \|(v_0,v_1)\|_{\mathcal{A}^s_{m}}\big), \label{decayrate5A3} \\
\|\big(|D|^{s- \mathtt{k}^+}u_t, |D|^{s- \mathtt{k}^+}v_t \big)(t,\cdot)\|_{L^2}& \lesssim (1+t)^{-\frac{n}{2(\mathtt{k}^+ -\delta)}(\frac{1}{m}-\frac{1}{2})- \frac{s+ \mathtt{k}^+ -4\delta}{2(\mathtt{k}^+ -\delta)}} \big(\|(u_0,u_1)\|_{\mathcal{A}^s_{m}}+ \|(v_0,v_1)\|_{\mathcal{A}^s_{m}}\big). \label{decayrate5A4}
\end{align}}

Finally, we obtain large regular solutions to (\ref{pt1.4}) by using the fractional Sobolev embedding.\medskip

\noindent\textbf{Theorem 6.}\textit{ Let $s_1 \ge s_2 > \frac{n}{2}+\mathtt{k}^+$, $s_1- s_2 \le \mathtt{k}^+$, $m \in [1,2)$ and $n> m_0 \mathtt{k}^-$. We assume the conditions
\begin{equation} \label{exponent6A1}
p > 1+ \max\Big\{ \frac{m(\mathtt{k}^+ +\sigma)}{n- m\mathtt{k}^-},\,s_1- \mathtt{k}^+,\,1\Big\} \text{ and } q > 1+ \max\Big\{\frac{2m\delta}{n},\,s_2- \mathtt{k}^+,\,1\Big\}.
\end{equation}
Moreover, we suppose the following conditions:
\begin{equation} \label{exponent6A2}
p \ge 1+ \frac{ms_1}{n- m\mathtt{k}^-} \text{ and } q \ge 1+ \frac{m(s_2+\mathtt{k}^- -2\sigma)}{n+ 2m(\mathtt{k}^+ -2\delta)}.
\end{equation}
Then, there exists a constant $\e>0$ such that for any small data
\begin{equation*}
\big((u_0,u_1),\, (v_0,v_1) \big) \in \mathcal{A}^{s_1}_{m} \times \mathcal{A}^{s_2}_{m} \text{ satisfying the assumption } \|(u_0,u_1)\|_{\mathcal{A}^{s_1}_{m}}+ \|(v_0,v_1)\|_{\mathcal{A}^{s_2}_{m}} \le \e,
\end{equation*}
we have a uniquely determined global (in time) small data energy solution
$$ (u,v) \in \Big(C([0,\ity),H^{s_1})\cap C^1([0,\ity),H^{s_1- \mathtt{k}^+})\Big) \times \Big(C([0,\ity),H^{s_2})\cap C^1([0,\ity),H^{s_2- \mathtt{k}^+})\Big) $$
to (\ref{pt1.1}).  Moreover, the estimates (\ref{decayrate3A1})-(\ref{decayrate3A6}) hold.}\medskip

\begin{vd}
\fontshape{n}
\selectfont
In the following examples, we fix $\sigma=\frac{3}{2}$, $\delta=\frac{1}{8}$ and $m=\frac{5}{4}$:
\begin{itemize}
\item If $n= 3$, then we may choose for example $p= 2,\, q \in \big(\frac{103}{26},\ity\big)$ by using Theorem 1-A and $p,\, q \in \big(\frac{103}{43},\ity\big)$ by using Theorem 1-B.
\item If $n= 1$ and $s_1=s_2= 1$, then we may choose for example $p= 6,\, q \in \big(\frac{71}{6},\ity\big)$ by using Theorem 2-A and $p,\, q \in \big(\frac{71}{11},\ity\big)$ by using Theorem 2-B.
\item If $n= 4$, $s_1= 2$ and $s_2= 3$, then using Theorem 3 we derive $p \in \big(\frac{119}{59},\ity\big)$ and $q \in \big(3,\ity\big)$.
\item If $n= 4$, $s_1= 4$ and $s_2= 5$, then using Theorem 4 we derive $p \in \big(\frac{7}{2},\ity\big)$ and $q \in \big(\frac{9}{2},\ity\big)$.
\item If $n= 4$ and $s= 5$, then using Theorem 5 we derive $p,\,q \in \big(\frac{9}{2},\ity\big)$.
\item If $n= 4$, $s_1= 5$ and $s_2= 4$, then using Theorem 6 we derive $p \in \big(\frac{9}{2},\ity\big)$ and $q \in \big(\frac{7}{2},\ity\big)$.
\end{itemize}
\end{vd}

\noindent\textbf{The organization of this paper} is as follows: In Section \ref{Linear estimates}, we present estimates for the solutions to (\ref{pt1.3}). In particular, we derive $(L^m \cap L^2)- L^2$ and $L^2- L^2$ estimates for solutions with $m\in [1,2)$ in the case $\delta=\frac{\sigma}{2}$, $\delta=(0,\frac{\sigma}{2})$ and $\delta=(\frac{\sigma}{2},\sigma)$, respectively, in Sections \ref{Linear estimates1}, \ref{Linear estimates2} and \ref{Linear estimates3}. In Section \ref{Semi-linear estimates}, we prove our global (in time) existence results to (\ref{pt1.1}), (\ref{pt1.2}) and (\ref{pt1.4}). We bring the optimality of our exponents if $\sigma$ and $\delta$ are integer numbers in Section \ref{Optimality}. Finally, we state some concluding remarks and open problems in Section \ref{ConcludeOpen}.

\section{Estimates for the solutions of the linear Cauchy problem} \label{Linear estimates}
Main goal of this section is to obtain $(L^m \cap L^2)- L^2$ and $L^2- L^2$ estimates for the solution and some its derivatives to (\ref{pt1.3}). These estimates play an fundamental role to prove the global (in time) existence results to (\ref{pt1.1}), (\ref{pt1.2}) and (\ref{pt1.4}) in the next section. First, using partial Fourier transformation to (\ref{pt1.3}) we obtain the following Cauchy problem for $\hat{w}(t,\xi):=F_{x\rightarrow \xi}\big(w(t,x)\big)$, $\hat{w}_0(\xi):=F_{x\rightarrow \xi}\big(w_0(x)\big)$ and $\hat{w}_1(\xi):=F_{x\rightarrow \xi}\big(w_1(x)\big)$:
\begin{equation}
\hat{w}_{tt}+ \mu |\xi|^{2\delta} \hat{w}_t+ |\xi|^{2\sigma} \hat{w}=0,\,\, \hat{w}(0,\xi)= \hat{w}_0(\xi),\,\, \hat{w}_t(0,\xi)= \hat{w}_1(\xi). \label{pt3.1}
\end{equation}
The characteristic roots are
$$ \lambda_{1,2}=\lambda_{1,2}(\xi)= \f{1}{2}\Big(-|\xi|^{2\delta}\pm \sqrt{|\xi|^{4\delta}-4|\xi|^{2\sigma}}\Big). $$
The solution to (\ref{pt3.1}) is presented by the following formula (here we assume $\lambda_{1}\neq \lambda_{2}$):
\begin{equation}
\hat{w}(t,\xi)= \frac{\lambda_1 e^{\lambda_2 t}-\lambda_2 e^{\lambda_1 t}}{\lambda_1- \lambda_2}\hat{w}_0(\xi)+ \frac{e^{\lambda_1 t}-e^{\lambda_2 t}}{\lambda_1- \lambda_2}\hat{w}_1(\xi)=: \hat{K_0}(t,\xi)\hat{w}_0(\xi)+\hat{K_1}(t,\xi)\hat{w}_1(\xi). \label{pt3.2}
\end{equation}
Taking account of the cases of small and large frequencies separately, we have the asymptotic behavior of the characteristic roots as follows:
\begin{align}
&1.\,\, \delta=\frac{\sigma}{2}:& &\lambda_{1,2}=\frac{1}{2}(-1 \pm i \sqrt{3}) |\xi|^{\sigma}, \label{pt3.3} \\ 
&2.\,\, \delta \in \big(0,\frac{\sigma}{2}\big):& &\lambda_1\sim -|\xi|^{2(\sigma- \delta)},\,\, \lambda_2\sim -|\xi|^{2\delta},\,\, \lambda_1-\lambda_2 \sim |\xi|^{2\delta} \text{ for small } |\xi|, \label{pt3.4} \\
& & &\text{and }\lambda_{1,2} \sim -|\xi|^{2\delta}\pm i|\xi|^\sigma,\,\, \lambda_1-\lambda_2 \sim i|\xi|^\sigma \text{ for large } |\xi|, \label{pt3.5} \\
&3.\,\, \delta \in \big(\frac{\sigma}{2},\sigma\big):& &\lambda_{1,2} \sim -|\xi|^{2\delta}\pm i|\xi|^\sigma, \lambda_1-\lambda_2 \sim i|\xi|^\sigma \text{ for small } |\xi|, \label{pt3.6} \\
& & &\text{and }\lambda_1\sim -|\xi|^{2(\sigma- \delta)},\,\, \lambda_2\sim -|\xi|^{2\delta},\,\, \lambda_1-\lambda_2 \sim |\xi|^{2\delta} \text{ for large } |\xi|. \label{pt3.7}
\end{align}
We now decompose the solution to (\ref{pt1.3}) into two parts localized separately to low and high frequencies, that is,
$$ w(t,x)= w_\chi(t,x)+ w_{1-\chi}(t,x), $$
where
$$w_\chi(t,x)=F^{-1}\big(\chi(|\xi|)\hat{w}(t,\xi)\big) \text{ and } w_{1-\chi}(t,x)=F^{-1}\big(\big(1-\chi(|\xi|)\big)\hat{w}(t,\xi)\big), $$
with a smooth cut-off function $\chi(|\xi|)$ equal to $1$ for small $|\xi|$ and vanishing for large $|\xi|$.

\subsection{The case $\delta=\frac{\sigma}{2}$} \label{Linear estimates1}
In oder to derive the $(L^m \cap L^2)- L^2$ estimates, on the one hand, we control the $L^2$ norm of the solution by the $L^m$ norm of the data for $t \in [1,\ity)$. On the other hand, for $t \in (0,1]$ we also obtain the $L^2-L^2$ estimates by using the suitable regularity of the data $w_0$ and $w_1$. We shall prove the following result.
\bmd \label{md3.1}
Let $\delta=\f{\sigma}{2}$ in (\ref{pt1.3}) and $m \in [1,2)$. The Sobolev solutions to (\ref{pt1.3}) satisfy the $(L^m \cap L^2)-L^2$ estimates
$$ \big\|\partial_t^j |D|^a w(t,\cdot)\big\|_{L^2} \lesssim (1+t)^{-\frac{n}{\sigma}(\frac{1}{m}-\frac{1}{2})- \frac{a}{\sigma}-j}\|w_0\|_{L^m \cap H^{a+j\sigma}}+ (1+t)^{1-\frac{n}{\sigma}(\frac{1}{m}-\frac{1}{2})- \frac{a}{\sigma}-j}\|w_1\|_{L^m \cap H^{[a+(j-1)\sigma]^+}}, $$
and the $L^2-L^2$ estimates
$$ \big\|\partial_t^j |D|^a w(t,\cdot)\big\|_{L^2} \lesssim (1+t)^{- \frac{a}{\sigma}-j}\|w_0\|_{H^{a+j\sigma}}+ (1+t)^{1- \frac{a}{\sigma}-j}\|w_1\|_{H^{[a+(j-1)\sigma]^+}}, $$
for any non-negative number $a$, $j=0,1$ and for all space dimenstions $n \ge 1$.
\emd

First, to prove Proposition \ref{md3.1} we shall show the following auxiliary estimates.
\bbd \label{LemmaOscillating}
Let $\alpha \in (0,\infty)$ and $r \in [1,\ity]$. Then, the following estimates hold for any $t>0$:
\begin{align*}
&\big\|F^{-1}\big(|\xi|^a e^{-c_1 |\xi|^{\alpha}t}\cos (c_2 |\xi|^{\alpha}t)\big)(t,\cdot)\big\|_{L^r}\lesssim t^{-\frac{a}{\alpha}-\frac{n}{\alpha}(1-\frac{1}{r})}, \\ 
&\big\|F^{-1}\big(|\xi|^a e^{-c_1 |\xi|^{\alpha}t}\sin (c_2 |\xi|^{\alpha}t)\big)(t,\cdot)\big\|_{L^r}\lesssim t^{-\frac{a}{\alpha}-\frac{n}{\alpha}(1-\frac{1}{r})},
\end{align*}
with $a \ge 0$ and $n \ge 1$. Here $c_1$ is a positive and $c_2 \neq 0$ is a real constant.
\ebd
\begin{proof}
For the proof of the first statement one can see Proposition $12$ in \cite{DuongKainaneReissig}. According to the treatment of Proposition $12$ in \cite{DuongKainaneReissig}, with minor modifications in the steps of the proofs we may conclude the second statement.
\end{proof}

\renewcommand{\proofname}{Proof of Proposition \ref{md3.1}.}
\begin{proof}
In the first step, by (\ref{pt3.3}) taking account of estimates for $\hat{K_0}$ and $\hat{K_1}$ we re-write these terms as follows:
\begin{align*}
&|\xi|^a \hat{K_0}(t,\xi)= |\xi|^a \Big( \cos\Big (\frac{\sqrt{3}}{2}|\xi|^\sigma t\Big)+\f{1}{\sqrt{3}} \sin \Big(\frac{\sqrt{3}}{2}|\xi|^\sigma t\Big)\Big) e^{-\frac{1}{2}|\xi|^\sigma t}, \\
&|\xi|^a \hat{K_1}(t,\xi)= t\, |\xi|^a \int_0^1 e^{\frac{1}{2}\big(-1+ i(2s-1)\sqrt{3}\big)|\xi|^\sigma t}ds.
\end{align*}
Hence, applying Young's convolution inequality and Lemma \ref{LemmaOscillating} we arrive at the following $L^2-L^m$ estimate for $t \in [1,\ity)$:
\begin{align}
&\big\||D|^a w(t,\cdot)\big\|_{L^2} \lesssim \big\|F^{-1}\big(|\xi|^a \hat{K_0}(t,\xi)\big)(t,\cdot)\big\|_{L^r}\, \|w_0\|_{L^m}+ \big\|F^{-1}\big(|\xi|^a \hat{K_1}(t,\xi)\big)(t,\cdot)\big\|_{L^r}\, \|w_1\|_{L^m} \nonumber \\
&\qquad \lesssim \Big(\Big\|F^{-1}\Big(|\xi|^a \cos\Big (\frac{\sqrt{3}}{2}|\xi|^\sigma t\Big) e^{-\frac{1}{2}|\xi|^\sigma t}\Big)(t,\cdot)\Big\|_{L^r}+ \Big\|F^{-1}\Big(|\xi|^a \sin \Big(\frac{\sqrt{3}}{2}|\xi|^\sigma t\Big) e^{-\frac{1}{2}|\xi|^\sigma t}\Big)(t,\cdot)\Big\|_{L^r}\Big) \|w_0\|_{L^m} \nonumber \\
&\qquad \quad + t \int_0^1 \Big\|F^{-1} \Big(|\xi|^a e^{\frac{1}{2}\big(-1+ i(2s-1)\sqrt{3}\big)|\xi|^\sigma t}\Big)(t,\cdot)\Big\|_{L^r}ds\,\, \|w_1\|_{L^m} \nonumber \\
&\qquad \lesssim t^{-\frac{n}{\sigma}(1-\frac{1}{r})-\frac{a}{\sigma}}\|w_0\|_{L^m}+ t^{1-\frac{n}{\sigma}(1-\frac{1}{r})-\frac{a}{\sigma}}\|w_1\|_{L^m}, \label{p3.1.1}
\end{align}
where $\frac{1}{r}+ \frac{1}{m}= \frac{3}{2}$. Moreover, we also get the following $L^2-L^2$ estimate for $t \in (0,1]$:
\begin{align}
&\big\||D|^a w(t,\cdot)\big\|_{L^2} \lesssim \big\|F^{-1}\big(\hat{K_0}(t,\xi)\big)(t,\cdot)\big\|_{L^1}\, \|w_0\|_{\dot{H}^a}+ \big\|F^{-1}\big(|\xi|^{\min\{a,\sigma\}} \hat{K_1}(t,\xi)\big)(t,\cdot)\big\|_{L^1}\, \|w_1\|_{\dot{H}^{[a-\sigma]^+}} \nonumber \\
&\qquad \lesssim \Big(\Big\|F^{-1}\Big( \cos\Big (\frac{\sqrt{3}}{2}|\xi|^\sigma t\Big) e^{-\frac{1}{2}|\xi|^\sigma t}\Big)(t,\cdot)\Big\|_{L^1}+ \Big\|F^{-1}\Big( \sin \Big(\frac{\sqrt{3}}{2}|\xi|^\sigma t\Big) e^{-\frac{1}{2}|\xi|^\sigma t}\Big)(t,\cdot)\Big\|_{L^1}\Big) \|w_0\|_{\dot{H}^a} \nonumber \\
&\qquad \quad + t \int_0^1 \Big\|F^{-1} \Big(|\xi|^{\min\{a,\sigma\}} e^{\frac{1}{2}\big(-1+ i(2s-1)\sqrt{3}\big)|\xi|^\sigma t}\Big)(t,\cdot)\Big\|_{L^1}ds\,\, \|w_1\|_{\dot{H}^{[a-\sigma]^+}} \nonumber \\
&\qquad \lesssim \|w_0\|_{H^{a}}+ t^{1-\frac{\min\{a,\sigma\}}{\sigma}}\|w_1\|_{H^{[a-\sigma]^+}} \lesssim \|w_0\|_{H^{a}}+ \|w_1\|_{H^{[a-\sigma]^+}}. \label{p3.1.2}
\end{align} 
From (\ref{p3.1.1}) and (\ref{p3.1.2}) we may conclude the desired statements in Propostion \ref{md3.1} with $j=0$. In the second step, in oder to estimate for some derivatives in time of $\hat{K_0}$ and $\hat{K_1}$ we note that
\begin{align*}
& |\xi|^a \partial_t\hat{K_0}(t,\xi)= -\f{2}{\sqrt{3}}|\xi|^{a+\sigma} \sin \Big(\frac{\sqrt{3}}{2}|\xi|^\sigma t\Big) e^{-\frac{1}{2} |\xi|^\sigma t}, \\
&|\xi|^a \partial_t \hat{K_1}(t,\xi)= |\xi|^a \Big(\cos \Big(\frac{\sqrt{3}}{2}|\xi|^\sigma t \Big)-\f{1}{\sqrt{3}} \sin \Big(\frac{\sqrt{3}}{2}|\xi|^\sigma t \Big)\Big) e^{-\frac{1}{2} |\xi|^\sigma t}.
\end{align*}
Then, applying again Young's convolution inequality, Lemma \ref{LemmaOscillating} and using the suitable regularity of the data we may conclude immediately all the statement in Propostion \ref{md3.1} with $j=1$. Hence, the proof of Propostion \ref{md3.1} is completed.
\end{proof}

\subsection{The case $\delta \in (0,\frac{\sigma}{2})$} \label{Linear estimates2}
In order to derive the $(L^m \cap L^2)- L^2$ estimates, on the one hand we control the $L^2$ norm of the low-frequency part of the solution by the $L^m$ norm of the data. On the other hand, its high-frequency part is estimated by using the $L^2-L^2$ estimates with the suitable regularity of the data $w_0$ and $w_1$. We shall prove the following result.
\bmd \label{md3.2.1}
Let $\delta \in (0,\f{\sigma}{2})$ in (\ref{pt1.3}) and $m \in [1,2)$. The Sobolev solutions to (\ref{pt1.3}) satisfy the $(L^m \cap L^2)-L^2$ estimates
\begin{align*}
\big\|\partial_t^j |D|^a w(t,\cdot)\big\|_{L^2} &\lesssim (1+t)^{-\frac{n}{2(\sigma-\delta)}(\frac{1}{m}-\frac{1}{2})- \frac{a+2j\delta}{2(\sigma-\delta)}}\|w_0\|_{L^m \cap H^{a+j\sigma}}\\ 
&\qquad \qquad + (1+t)^{1-\frac{n}{2(\sigma-\delta)}(\frac{1}{m}-\frac{1}{2})- \frac{a+2j\delta}{2(\sigma-\delta)}}\|w_1\|_{L^m \cap H^{[a+(j-1)\sigma]^+}},
\end{align*}
and the $L^2-L^2$ estimates
$$ \big\|\partial_t^j |D|^a w(t,\cdot)\big\|_{L^2} \lesssim (1+t)^{-\frac{a+2j\delta}{2(\sigma-\delta)}}\|w_0\|_{H^{a+j\sigma}}+ (1+t)^{1- \frac{a+2j\delta}{2(\sigma-\delta)}}\|w_1\|_{H^{[a+(j-1)\sigma]^+}}, $$
for any non-negative number $a$,  $j=0,1$ and for all space dimenstions $n \ge 1$.
\emd

\renewcommand{\proofname}{Proof.}
\begin{proof}
We will divide our considerations into two steps. In the first step, let us devote to estimates for small frequencies. First, let us define $m'$ by $\frac{1}{m}+\frac{1}{m'}=1$ and recall the abbreviation $m_0=\frac{2m}{2-m}$, that is, $\frac{1}{m_0}= \frac{1}{m}- \frac{1}{2}$. Then by using the formula of Parseval-Plancherel and H\"{o}lder's inequality we obtain the following estimate:
\begin{align}
&\big\|\partial_t^j |D|^a w_\chi(t,\cdot)\big\|_{L^2}= \big\| |\xi|^a \chi(\xi) \partial_t^j \hat{w}(t,\xi)\big\|_{L^2} \nonumber \\ 
&\qquad \lesssim \big\| |\xi|^a \chi(\xi) \partial_t^j \hat{K_0}(t,\xi)\big\|_{L^{m_0}} \|\hat{w}_0\|_{L^{m'}}+ \big\| |\xi|^a \chi(\xi) \partial_t^j \hat{K_1}(t,\xi)\big\|_{L^{m_0}} \|\hat{w}_1\|_{L^{m'}}. \label{p2.2.1}
\end{align}
We can control $\|\hat{w}_0\|_{L^{m'}}$ and $\|\hat{w}_1\|_{L^{m'}}$, respectively, by $\|w_0\|_{L^m}$ and $\|w_1\|_{L^m}$. Hence, we have only to estimate the $L^{m_0}$ norm of the multipliers. Taking account of estimates for $\hat{K_j}$ with $j=0,1$ and some their derivatives we re-write these terms for small frequencies as follows:
\begin{align}
&\hat{K_1}(t,\xi)= e^{\lambda_1 t}\f{1- e^{(\lambda_2- \lambda_1)t}}{\lambda_1-\lambda_2}= t e^{\lambda_1 t}\int_0^1 e^{-r(\lambda_1-\lambda_2)t}dr, \label{p2.2.2} \\ 
&\hat{K_0}(t,\xi)= -\lambda_1 \hat{K_1}+e^{\lambda_1 t},\,\, \partial_t \hat{K_1}(t,\xi)= \hat{K_0}+ (\lambda_1+\lambda_2)\hat{K_1} \text{ and } \partial_t \hat{K_0}(t,\xi)= -\lambda_1\lambda_2\hat{K_1}. \label{p2.2.3}
\end{align}
For the sake of the asymptotic behavior of the characteristic roots in (\ref{pt3.4}), we arrive at
\begin{align}
&\big|\chi(\xi) \hat{K_1}(t,\xi)\big| \lesssim t e^{-c |\xi|^{2(\sigma- \delta)}t},\,\, \big|\chi(\xi) \hat{K_0}(t,\xi)\big| \lesssim e^{-c |\xi|^{2(\sigma- \delta)}t}, \label{p2.2.4} \\
&\big|\chi(\xi) \partial_t \hat{K_1}(t,\xi)\big| \lesssim (1+t |\xi|^{2\delta})e^{-c |\xi|^{2(\sigma- \delta)}t},\,\, \big|\chi(\xi) \partial_t \hat{K_0}(t,\xi) \big| \lesssim t |\xi|^{2\sigma}e^{-c |\xi|^{2(\sigma- \delta)}t}, \label{p2.2.5}
\end{align}
where $c$ is a suitable positive constant. We can see that it holds for small frequencies
\begin{equation} \label{L1normEstimate}
\int_{\R^n} |\xi|^\beta e^{-c|\xi|^\alpha t}d\xi \lesssim (1+t)^{-\frac{n+\beta}{\alpha}},
\end{equation}
for any $n \ge 1$, $\beta \in \R$ satisfying $n+\beta >0$ and for all positive numbers $c,\,\alpha >0$. Hence, from (\ref{p2.2.4}) and (\ref{p2.2.5}) we may conclude immediately the following estimates:
\begin{align}
&\big\| |\xi|^a \chi(\xi) \partial_t^j \hat{K_0}(t,\xi)\big\|_{L^{m_0}} \lesssim (1+t)^{-\frac{n+(\alpha+2\delta j)m_0}{2(\sigma-\delta)m_0}}, \label{p2.2.6} \\ 
&\big\| |\xi|^a \chi(\xi) \partial_t^j \hat{K_1}(t,\xi)\big\|_{L^{m_0}} \lesssim (1+t)^{1-\frac{n+(\alpha+2\delta j)m_0}{2(\sigma-\delta)m_0}}. \label{p2.2.7}
\end{align}
Therefore, from (\ref{p2.2.1}), (\ref{p2.2.6}) and (\ref{p2.2.7}) we have proved that
\begin{equation} \label{p2.2.8}
\big\|\partial_t^j |D|^a w_\chi(t,\cdot)\big\|_{L^2} \lesssim (1+t)^{-\frac{n}{2(\sigma-\delta)}(\frac{1}{m}-\frac{1}{2})- \frac{a+2j\delta}{2(\sigma-\delta)}} \|w_0\|_{L^m}+ (1+t)^{1-\frac{n}{2(\sigma-\delta)}(\frac{1}{m}-\frac{1}{2})- \frac{a+2j\delta}{2(\sigma-\delta)}} \|w_1\|_{L^m}.
\end{equation}
Next, let us turn to estimate the solution and some its derivatives to (\ref{pt1.3}) for large frequencies. Thanks to the asymptotic behavior of the characteristic roots in (\ref{pt3.5}), we find
\begin{align*}
&\big|\big(1-\chi(\xi)\big)\hat{K_0}(t,\xi)\big| \lesssim e^{-c |\xi|^{2\delta}t},\,\, \big|\big(1-\chi(\xi)\big)\hat{K_1}(t,\xi)\big| \lesssim |\xi|^{-\sigma}e^{-c |\xi|^{2\delta}t}, \\
&\big|\big(1-\chi(\xi)\big)\partial_t \hat{K_0}(t,\xi)\big| \lesssim |\xi|^{\sigma}e^{-c |\xi|^{2\delta}t},\,\, \big|\big(1-\chi(\xi)\big)\partial_t \hat{K_1}(t,\xi) \big| \lesssim e^{-c |\xi|^{2\delta}t}.
\end{align*}
By applying again the formula of Parseval-Plancherel and using the suitable regularity of the data $w_0$ and $w_1$ we derive the following estimates:
\begin{equation} \label{p2.2.9}
\big\|\partial_t^j |D|^a w_{1-\chi}(t,\cdot)\big\|_{L^2} \lesssim \|w_0\|_{H^{a+j\sigma}}+ \|w_1\|_{H^{[a+(j-1)\sigma]^+}}.
\end{equation}
Summarizing, from (\ref{p2.2.8}) and (\ref{p2.2.9}) we may conclude that the proof of Proposition \ref{md3.2.1} is completed.
\end{proof}

\begin{nx}
\fontshape{n}
\selectfont
In Proposition \ref{md3.2.1} we state estimates for the solution and some its derivatives to (\ref{pt1.3}) which hold for any space dimensions $n \ge 1$. Moreover, we may prove a better result under a restriction to space dimensions $n>2m_0 \delta$. We obtain the following sharper estimates.
\end{nx}

\bmd \label{md3.2.2}
Let $\delta \in (0,\f{\sigma}{2})$ in (\ref{pt1.3}) and $m \in [1,2)$. The Sobolev solutions to (\ref{pt1.3}) satisfy the $(L^m \cap L^2)-L^2$ estimates
\begin{align*}
\big\|\partial_t^j |D|^a w(t,\cdot)\big\|_{L^2} &\lesssim (1+t)^{-\frac{n}{2(\sigma-\delta)}(\frac{1}{m}-\frac{1}{2})- \frac{a}{2(\sigma-\delta)}-j}\|w_0\|_{L^m \cap H^{a+j\sigma}}\\ 
&\qquad \qquad + (1+t)^{-\frac{n}{2(\sigma-\delta)}(\frac{1}{m}-\frac{1}{2})- \frac{a-2\delta}{2(\sigma-\delta)}-j}\|w_1\|_{L^m \cap H^{[a+(j-1)\sigma]^+}},
\end{align*}
and the $L^2-L^2$ estimates
$$ \big\|\partial_t^j |D|^a w(t,\cdot)\big\|_{L^2} \lesssim (1+t)^{-\frac{a}{2(\sigma-\delta)}-j}\|w_0\|_{H^{a+j\sigma}}+ (1+t)^{- \frac{a-2\delta}{2(\sigma-\delta)}-j}\|w_1\|_{H^{[a+(j-1)\sigma]^+}}, $$
for any non-negative number $a$, $j=0,1$ and for all space dimenstions $n>2m_0 \delta$.
\emd

\begin{proof}
The proof of this proposition is similar to the proof to Proposition \ref{md3.2.1}. The shaper estimates appearing Proposition \ref{md3.2.2} in comparison with Proposition \ref{md3.2.1} rely on estimates for small frequencies. For this reason, we only present the steps of the proofs for small $|\xi|$. First, we notice that we do not use the relation (\ref{p2.2.2}) as we did in the proof of Proposition \ref{md3.2.1}. By (\ref{pt3.4}) and (\ref{p2.2.3}), we get immediately the following estimates:
\begin{align*}
&\big|\chi(\xi) \hat{K_1}(t,\xi)\big| \lesssim |\xi|^{-2\delta} e^{-c |\xi|^{2(\sigma- \delta)}t},\,\, \big|\chi(\xi) \hat{K_0}(t,\xi)\big| \lesssim e^{-c |\xi|^{2(\sigma- \delta)}t}, \\
&\big|\chi(\xi) \partial_t \hat{K_1}(t,\xi)\big| \lesssim |\xi|^{2(\sigma-2\delta)}e^{-c |\xi|^{2(\sigma- \delta)}t}+ e^{-c |\xi|^{2\delta}t},\,\, \big|\chi(\xi) \partial_t \hat{K_0}(t,\xi) \big| \lesssim |\xi|^{2(\sigma-\delta)}e^{-c |\xi|^{2(\sigma- \delta)}t},
\end{align*}
where $c$ is a suitable positive constant. Then, by using again (\ref{L1normEstimate}) and the condition $n>2m_0 \delta$ we shall repeat some of the arguments as we did in the proof to Proposition \ref{md3.2.1} to conclude all the desired estimates. Hence, this completes the proof of Proposition \ref{md3.2.2}.
\end{proof}

\subsection{The case $\delta \in (\frac{\sigma}{2},\sigma)$} \label{Linear estimates3}
In this section, we will apply the same strategy as we did in Section \ref{Linear estimates2} with minor modifications in the steps of the proofs. We obtain the following results.

\bmd \label{md3.3.1}
Let $\delta \in (\f{\sigma}{2},\sigma)$ in (\ref{pt1.3}) and $m \in [1,2)$. The Sobolev solutions to (\ref{pt1.3}) satisfy the $(L^m \cap L^2)-L^2$ estimates
\begin{align*}
\big\|\partial_t^j |D|^a w(t,\cdot)\big\|_{L^2} &\lesssim (1+t)^{-\frac{n}{2\delta}(\frac{1}{m}-\frac{1}{2})- \frac{a+j\sigma}{2\delta}}\|w_0\|_{L^m \cap H^{a+2j(\sigma-\delta)}}\\ 
&\qquad \qquad + (1+t)^{1-\frac{n}{2\delta}(\frac{1}{m}-\frac{1}{2})- \frac{a+j\sigma}{2\delta}}\|w_1\|_{L^m \cap H^{[a+2(j-1)\delta]^+}},
\end{align*}
and the $L^2-L^2$ estimates
$$ \big\|\partial_t^j |D|^a w(t,\cdot)\big\|_{L^2} \lesssim (1+t)^{- \frac{a+j\sigma}{2\delta}}\|w_0\|_{H^{a+2j(\sigma-\delta)}}+ (1+t)^{1- \frac{a+j\sigma}{2\delta}}\|w_1\|_{H^{[a+2(j-1)\delta]^+}}, $$
for any non-negative number $a$, $j=0,1$ and for all space dimenstions $n \ge 1$.
\emd

\begin{proof}
The proof of this proposition is similar to the proof to Proposition \ref{md3.2.1}. Hence, it is reasonable to present only the steps which are different. First, we re-write $\hat{K_1}$ for small frequencies as follows:
$$ \hat{K_1}(t,\xi)= t e^{\lambda_1 t}\int_0^1 e^{-i r \sqrt{4|\xi|^{2\sigma}-|\xi|^{4\delta}}t}dr. $$
Using the relation (\ref{p2.2.3}) and the asymptotic behavior of the characteristic roots in (\ref{pt3.6}) we have
\begin{align*}
&\big|\chi(\xi) \hat{K_1}(t,\xi)\big| \lesssim t e^{-c |\xi|^{2\delta}t},\,\, \big|\chi(\xi) \hat{K_0}(t,\xi)\big| \lesssim e^{-c |\xi|^{2\delta}t}, \\
&\big|\chi(\xi) \partial_t \hat{K_1}(t,\xi)\big| \lesssim (1+t |\xi|^{\sigma})e^{-c |\xi|^{2\delta}t},\,\, \big|\chi(\xi) \partial_t \hat{K_0}(t,\xi) \big| \lesssim |\xi|^{\sigma}e^{-c |\xi|^{2\delta}t},
\end{align*}
where $c$ is a suitable positive constant. Moreover, by the asymptotic behavior of the characteristic roots in (\ref{pt3.7}) we arrive at
\begin{align*}
&\big|\big(1-\chi(\xi)\big)\hat{K_0}(t,\xi)\big| \lesssim e^{-c |\xi|^{2(\sigma-\delta)}t},\,\, \big|\big(1-\chi(\xi)\big)\hat{K_1}(t,\xi)\big| \lesssim |\xi|^{-2\delta}e^{-c |\xi|^{2(\sigma-\delta)}t}, \\
&\big|\big(1-\chi(\xi)\big)\partial_t \hat{K_0}(t,\xi)\big| \lesssim |\xi|^{2(\sigma-\delta)}e^{-c |\xi|^{2(\sigma-\delta)}t},\,\, \big|\big(1-\chi(\xi)\big)\partial_t \hat{K_1}(t,\xi) \big| \lesssim e^{-c |\xi|^{2(\sigma-\delta)}t}.
\end{align*}
Finally, repeating some of the arguments as we did in the proof to Proposition \ref{md3.2.1} we may conclude all the statements in Proposition \ref{md3.3.1}. Summarizing, the proof of Proposition \ref{md3.3.1} is completed.
\end{proof}

\begin{nx}
\fontshape{n}
\selectfont
We can see that in Proposition \ref{md3.3.1} we state estimates for the solution and some its derivatives to (\ref{pt1.3}) which hold for any space dimensions $n \ge 1$. Moreover, we may prove a better result under a restriction to space dimensions $n>m_0 \sigma$. We obtain the following sharper estimates.
\end{nx}

\bmd \label{md3.3.2}
Let $\delta \in (\f{\sigma}{2},\sigma)$ in (\ref{pt1.3}) and $m \in [1,2)$. The Sobolev solutions to (\ref{pt1.3}) satisfy the $(L^m \cap L^2)-L^2$ estimates
\begin{align*}
\big\|\partial_t^j |D|^a w(t,\cdot)\big\|_{L^2} &\lesssim (1+t)^{-\frac{n}{2\delta}(\frac{1}{m}-\frac{1}{2})- \frac{a+j\sigma}{2\delta}}\|w_0\|_{L^m \cap H^{a+2j(\sigma-\delta)}} \\ 
&\qquad \qquad + (1+t)^{-\frac{n}{2\delta}(\frac{1}{m}-\frac{1}{2})- \frac{a+(j-1)\sigma}{2\delta}}\|w_1\|_{L^m \cap H^{[a+2(j-1)\delta]^+}},
\end{align*}
and the $L^2-L^2$ estimates
$$ \big\|\partial_t^j |D|^a w(t,\cdot)\big\|_{L^2} \lesssim (1+t)^{- \frac{a+j\sigma}{2\delta}}\|w_0\|_{H^{a+2j(\sigma-\delta)}}+ (1+t)^{- \frac{a+(j-1)\sigma}{2\delta}}\|w_1\|_{H^{[a+2(j-1)\delta]^+}}, $$
for any non-negative number $a$, $j=0,1$ and for all space dimenstions $n>m_0 \sigma$.
\emd

\begin{proof}
The proof of this proposition is similar to the proof to Proposition \ref{md3.3.1}. The shaper estimates appearing Proposition \ref{md3.3.2} in comparison with Proposition \ref{md3.3.1} rely on estimates for small frequencies. For this reason, we only present the steps of the proofs for small $|\xi|$. For the sake of the asymptotic behavior of the characteristic roots in (\ref{pt3.6}), we get
\begin{align*}
&\big|\chi(\xi) \hat{K_1}(t,\xi)\big| \lesssim |\xi|^{-\sigma} e^{-c |\xi|^{2\delta}t},\,\, \big|\chi(\xi) \hat{K_0}(t,\xi)\big| \lesssim e^{-c |\xi|^{2\delta}t}, \\
&\big|\chi(\xi) \partial_t \hat{K_1}(t,\xi)\big| \lesssim e^{-c |\xi|^{2\delta}t},\,\, \big|\chi(\xi) \partial_t \hat{K_0}(t,\xi) \big| \lesssim |\xi|^{\sigma}e^{-c |\xi|^{2\delta}t},
\end{align*}
where $c$ is a suitable positive constant. By using again (\ref{L1normEstimate}) and the condition $n>m_0 \sigma$, repeating some of the arguments as we did in the proof to Proposition \ref{md3.3.1} we may conclude all the desired estimates. Therefore, Proposition \ref{md3.3.2} is proved.
\end{proof}

\noindent From the statements of Propositions \ref{md3.1}, \ref{md3.2.1} and \ref{md3.3.1}, we conclude the following corollary.
\begin{hq} \label{hq3.1}
Let $\delta \in (0,\sigma)$ in (\ref{pt1.3}) and $m \in [1,2)$. The solution to (\ref{pt1.3}) satisfies the $(L^m \cap L^2)-L^2$ estimates
\begin{align*}
\big\|\partial_t^j |D|^a w(t,\cdot)\big\|_{L^2} &\lesssim (1+t)^{-\frac{n}{2(\mathtt{k}^+ -\delta)}(\frac{1}{m}-\frac{1}{2})- \frac{a+ j\mathtt{k}^-}{2(\mathtt{k}^+ -\delta)}}\|w_0\|_{L^m \cap H^{a+j(2\sigma- \mathtt{k}^+)}} \\ 
&\qquad \qquad + (1+t)^{1-\frac{n}{2(\mathtt{k}^+ -\delta)}(\frac{1}{m}-\frac{1}{2})- \frac{a+ j\mathtt{k}^-}{2(\mathtt{k}^+ -\delta)}}\|w_1\|_{L^m \cap H^{[a+ (j-1)\mathtt{k}^+]^+}},
\end{align*}
and the $L^2-L^2$ estimates
$$ \big\|\partial_t^j |D|^a w(t,\cdot)\big\|_{L^2} \lesssim (1+t)^{- \frac{a+j \mathtt{k}^-}{2(\mathtt{k}^+ -\delta)}}\|w_0\|_{H^{a+j(2\sigma- \mathtt{k}^+)}}+ (1+t)^{1- \frac{a+ j\mathtt{k}^-}{2(\mathtt{k}^+ -\delta)}}\|w_1\|_{H^{[a+ (j-1)\mathtt{k}^+]^+}}, $$
for any non-negative number $a$, $j=0,1$ and for all space dimenstions $n \ge 1$.
\end{hq}

\noindent Finally, from the statements of Propositions \ref{md3.1}, \ref{md3.2.2} and \ref{md3.3.2} we obtain the following sharper estimates under a restriction to space dimensions $n>m_0 \mathtt{k}^-$.
\begin{hq} \label{hq3.2}
Let $\delta \in (0,\sigma)$ in (\ref{pt1.3}) and $m \in [1,2)$. The solution to (\ref{pt1.3}) satisfies the $(L^m \cap L^2)-L^2$ estimates
\begin{align*}
\big\|\partial_t^j |D|^a w(t,\cdot)\big\|_{L^2} &\lesssim (1+t)^{-\frac{n}{2(\mathtt{k}^+ -\delta)}(\frac{1}{m}-\frac{1}{2})- \frac{a+ j(2\sigma- \mathtt{k}^-)}{2(\mathtt{k}^+ -\delta)}}\|w_0\|_{L^m \cap H^{a+j(2\sigma- \mathtt{k}^+)}} \\ 
&\qquad \qquad + (1+t)^{-\frac{n}{2(\mathtt{k}^+ -\delta)}(\frac{1}{m}-\frac{1}{2})- \frac{a+ j(2\sigma- \mathtt{k}^-)- \mathtt{k}^-}{2(\mathtt{k}^+ -\delta)}}\|w_1\|_{L^m \cap H^{[a+ (j-1)\mathtt{k}^+]^+}},
\end{align*}
and the $L^2-L^2$ estimates
$$ \big\|\partial_t^j |D|^a w(t,\cdot)\big\|_{L^2} \lesssim (1+t)^{- \frac{a+ j(2\sigma- \mathtt{k}^-)}{2(\mathtt{k}^+ -\delta)}}\|w_0\|_{L^m \cap H^{a+j(2\sigma- \mathtt{k}^+)}}+ (1+t)^{- \frac{a+ j(2\sigma- \mathtt{k}^-)- \mathtt{k}^-}{2(\mathtt{k}^+ -\delta)}}\|w_1\|_{L^m \cap H^{[a+ (j-1)\mathtt{k}^+]^+}}, $$
for any non-negative number $a$, $j=0,1$ and for all space dimenstions $n>m_0 \mathtt{k}^-$.
\end{hq}

\begin{nx}
\fontshape{n}
\selectfont
The statements in Corollaries \ref{hq3.1} and \ref{hq3.2} are key tools to prove global (in time) existence results for the weakly coupled systems of semi-linear models (\ref{pt1.1}), (\ref{pt1.2}) and (\ref{pt1.4}). Here we want to underline that the decay estimates for solution and some its derivatives to  (\ref{pt1.3}) from Corollary \ref{hq3.2} are better than those from Corollary \ref{hq3.1}. For this reason, in the next section we only present the steps of the proofs to our global (in time) existence results in detail by using all statements from Corollary \ref{hq3.2}.
\end{nx}

\section{Treatment of weakly coupled systems of corresponding semi-linear models} \label{Semi-linear estimates}

\subsection{Philosophy of our approach}
In this section, we will apply the estimates for the solutions to (\ref{pt1.3}) from Corollary \ref{hq3.2} to prove the global (in time) existence of small data Sobolev solutions to weakly coupled systems of semi-linear models (\ref{pt1.1}), (\ref{pt1.2}) and (\ref{pt1.4}). Recalling the fundamental solutions $K_0$ and $K_1$ defined in Section \ref{Linear estimates} we write the solutions of the corresponding linear Cauchy problems with vanishing right-hand sides to (\ref{pt1.1}) and (\ref{pt1.2}) and (\ref{pt1.4}) in the following form:
$$\begin{cases}
u^{ln}(t,x)=K_0(t,x) \ast_{x} u_0(x)+ K_1(t,x) \ast_{x} u_1(x), \\ 
v^{ln}(t,x)=K_0(t,x) \ast_{x} v_0(x)+ K_1(t,x) \ast_{x} v_1(x).
\end{cases}$$
Applying Duhamel's principle gives the formal implicit representation of the solutions to (\ref{pt1.1}), (\ref{pt1.2}) and (\ref{pt1.4}) as follows:
$$\begin{cases}
u(t,x)= u^{ln}(t,x) + \int_0^t K_1(t-\tau,x) \ast_x f(v,v_t) d\tau=: u^{ln}(t,x)+ u^{nl}(t,x), \\ 
v(t,x)= v^{ln}(t,x) + \int_0^t K_1(t-\tau,x) \ast_x f(u,u_t) d\tau=: v^{ln}(t,x)+ v^{nl}(t,x).
\end{cases}$$
Here $f(v,v_t)=|v(t,x)|^p$ and $f(u,u_t)=|u(t,x)|^q$, $f(v,v_t)=|v_t(t,x)|^p$ and $f(u,u_t)=|u_t(t,x)|^q$, $f(v,v_t)=|v(t,x)|^p$ and $f(u,u_t)=|u_t(t,x)|^q$, respectively, to (\ref{pt1.1}), to (\ref{pt1.2}),  to (\ref{pt1.4}).\\
We choose the data spaces $(u_0,u_1) \in \mathcal{A}^{s_1}_m$ and $(v_0,v_1) \in \mathcal{A}^{s_2}_m$. Moreover, we introduce the family $\{X(t)\}_{t>0}$ of solution spaces $X(t)$ with the norm
\begin{align*}
\|(u,v)\|_{X(t)}:= \sup_{0\le \tau \le t} \Big( f_{1}(\tau)^{-1}\|u(\tau,\cdot)\|_{L^2} &+ f_{1,s_1}(\tau)^{-1}\big\||D|^{s_1} u(\tau,\cdot)\big\|_{L^2}\\
&+ f_{2}(\tau)^{-1}\|u_t(\tau,\cdot)\|_{L^2}+ f_{2,s_1}(\tau)^{-1}\big\||D|^{s_1- \mathtt{k}^+} u_t(\tau,\cdot)\big\|_{L^2}\\
g_{1}(\tau)^{-1}\|v(\tau,\cdot)\|_{L^2} &+ g_{1,s_2}(\tau)^{-1}\big\||D|^{s_2} v(\tau,\cdot)\big\|_{L^2}\\
&+ g_{2}(\tau)^{-1}\|v_t(\tau,\cdot)\|_{L^2}+ g_{2,s_2}(\tau)^{-1}\big\||D|^{s_2- \mathtt{k}^+} v_t(\tau,\cdot)\big\|_{L^2} \Big),
\end{align*}
where
\begin{align}
&f_{1}(\tau)= g_{1}(\tau)= (1+\tau)^{-\frac{n}{2(\mathtt{k}^+ -\delta)}(\frac{1}{m}-\frac{1}{2})+ \frac{\mathtt{k}^-}{2(\mathtt{k}^+ -\delta)}},\,\,\, f_{1,s_1}(\tau)=(1+\tau)^{-\frac{n}{2(\mathtt{k}^+ -\delta)}(\frac{1}{m}-\frac{1}{2})- \frac{s_1- \mathtt{k}^-}{2(\mathtt{k}^+ -\delta)}}, \label{pt4.1}\\
&f_{2}(\tau)=g_{2}(\tau)= (1+\tau)^{-\frac{n}{2(\mathtt{k}^+ -\delta)}(\frac{1}{m}-\frac{1}{2})- \frac{\sigma- \mathtt{k}^-}{\mathtt{k}^+ -\delta}},\,\,\, f_{2,s_1}(\tau)=(1+\tau)^{-\frac{n}{2(\mathtt{k}^+ -\delta)}(\frac{1}{m}-\frac{1}{2})- \frac{s_1+ \mathtt{k}^+ -4\delta}{2(\mathtt{k}^+ -\delta)}}, \label{pt4.2} \\
&g_{1,s_2}(\tau)=(1+\tau)^{-\frac{n}{2(\mathtt{k}^+ -\delta)}(\frac{1}{m}-\frac{1}{2})- \frac{s_2- \mathtt{k}^-}{2(\mathtt{k}^+ -\delta)}},\,\,\, g_{2,s_2}(\tau)=(1+\tau)^{-\frac{n}{2(\mathtt{k}^+ -\delta)}(\frac{1}{m}-\frac{1}{2})- \frac{s_2+ \mathtt{k}^+ -4\delta}{2(\mathtt{k}^+ -\delta)}}. \label{pt4.3}
\end{align}
We define for all $t>0$ the operator $N: \quad (u,v) \in X(t) \longrightarrow N(u,v) \in X(t)$ by the formula
$$N(u,v)(t,x)= \big(u^{ln}(t,x)+ u^{nl}(t,x), v^{ln}(t,x)+ v^{nl}(t,x)\big). $$
We will prove that the operator $N$ satisfies the following two inequalities:
\begin{align}
&\|N(u,v)\|_{X(t)} \lesssim \|(u_0,u_1)\|_{\mathcal{A}^{s_1}_m}+ \|(v_0,v_1)\|_{\mathcal{A}^{s_2}_m}+ \|(u,v)\|^p_{X(t)}+ \|(u,v)\|^q_{X(t)}, \label{pt4.3}\\
&\|N(u,v)-N(\bar{u},\bar{v})\|_{X(t)} \nonumber\\
&\qquad \qquad \qquad \lesssim \|(u,v)-(\bar{u},\bar{v})\|_{X(t)} \Big(\|(u,v)\|^{p-1}_{X(t)}+ \|(\bar{u},\bar{v})\|^{p-1}_{X(t)}+ \|(u,v)\|^{q-1}_{X(t)}+ \|(\bar{u},\bar{v})\|^{q-1}_{X(t)}\Big). \label{pt4.4}
\end{align}
Then, applying Banach's fixed point theorem we obtain local (in time) existence results of large data solutions and global (in time) existence results of small data solutions as well.
\begin{nx}
\fontshape{n}
\selectfont
From the definition of the norm in $X(t)$, by replacing $a=s_1$ and $a=s_2$ in the statements from Corollary \ref{hq3.2} we may conclude 
\begin{equation}
\big\|(u^{ln}, v^{ln})\big\|_{X(t)} \lesssim \|(u_0,u_1)\|_{\mathcal{A}^{s_1}_{m}}+ \|(v_0,v_1)\|_{\mathcal{A}^{s_2}_{m}}, \text{ for all }s_1 \text{ and }s_2 \ge 0. \label{pt4.5}
\end{equation}
Hence, in order to complete the proof of (\ref{pt4.3}) it is reasonable to prove the following inequality:
\begin{equation}
\big\|(u^{nl}, v^{nl})\big\|_{X(t)} \lesssim \|(u,v)\|^p_{X(t)}+ \|(u,v)\|^q_{X(t)}. \label{pt4.31}
\end{equation}
\end{nx}
Now we are going to prove our main results from Section \ref{Sec.main}. Without loss of generality, we can assume $q>p$ throughout the following proofs.

\subsection{Proof of Theorem 1-A: $s_1=s_2=\mathtt{k}^+$}
We introduce the solution space
$$X(t):= \Big(C([0,t],H^{\mathtt{k}^+}) \cap C^1([0,t],L^2)\Big)^2, $$
where the weights are modified in the following way:
\begin{align*}
&f_{1}(\tau)= (1+\tau)^{-\frac{n}{2(\mathtt{k}^+ -\delta)}(\frac{1}{m}-\frac{1}{2})+ \frac{\mathtt{k}^-}{2(\mathtt{k}^+ -\delta)}+ \e(p)},\,\,\, f_{1,\mathtt{k}^+}(\tau)=(1+\tau)^{-\frac{n}{2(\mathtt{k}^+ -\delta)}(\frac{1}{m}-\frac{1}{2})- \frac{\mathtt{k}^+ - \mathtt{k}^-}{2(\mathtt{k}^+ -\delta)}+ \e(p)}, \\
&f_{2}(\tau)= (1+\tau)^{-\frac{n}{2(\mathtt{k}^+ -\delta)}(\frac{1}{m}-\frac{1}{2})- \frac{\sigma- \mathtt{k}^-}{\mathtt{k}^+ -\delta}+ \e(p)},\,\,\,  f_{2,s_1}(\tau)=g_{2,s_2}(\tau) \equiv 0.
\end{align*}
First, let us prove the inequality (\ref{pt4.31}). In order to control some estimates for $u^{nl}$, our strategy is to use the $(L^m \cap L^2)- L^2$ estimates if $\tau \in [0,t/2]$ and the $L^2-L^2$ estimates if $\tau \in [t/2,t]$ from Corollary \ref{hq3.2}. Hence, we have the following estimates for $j,l=0,1$ and $(j,l) \neq (1,1)$:
\begin{align*}
\big\|\partial_t^j |D|^{l\mathtt{k}^+} u^{nl}(t,\cdot)\big\|_{L^2} &\lesssim \int_0^{t/2}(1+t-\tau)^{-\frac{n}{2(\mathtt{k}^+ -\delta)}(\frac{1}{m}-\frac{1}{2})- \frac{l\mathtt{k}^+ + j(2\sigma- \mathtt{k}^-)- \mathtt{k}^-}{2(\mathtt{k}^+ -\delta)}}\big\||v(\tau,\cdot)|^p\big\|_{L^m \cap L^2}d\tau\\
&\qquad + \int_{t/2}^t (1+t-\tau)^{- \frac{l\mathtt{k}^+ + j(2\sigma- \mathtt{k}^-)- \mathtt{k}^-}{2(\mathtt{k}^+ -\delta)}}\big\||v(\tau,\cdot)|^p\big\|_{L^2}d\tau.
\end{align*}
For this reason, we need to estimate for $|v(\tau,x)|^p$ in $L^m \cap L^2$ and $L^2$ as follows:
$$\big\||v(\tau,\cdot)|^p\big\|_{L^m \cap L^2} \lesssim \|v(\tau,\cdot)\|^p_{L^{mp}}+ \|v(\tau,\cdot)\|^p_{L^{2p}},\, \text{ and }\big\||v(\tau,\cdot)|^p\big\|_{L^2}= \|v(\tau,\cdot)\|^p_{L^{2p}}.$$
Employing the fractional Gagliardo-Nirenberg inequality from Proposition \ref{fractionalGagliardoNirenberg} gives
\begin{align*}
\big\||v(\tau,\cdot)|^p\big\|_{L^m \cap L^2} &\lesssim (1+\tau)^{-\frac{n}{2m(\mathtt{k}^+ -\delta)}(p-1)+ \frac{p\mathtt{k}^-}{2(\mathtt{k}^+ -\delta)}}\|(u,v)\|^p_{X(\tau)},\\
\big\||v(\tau,\cdot)|^p\big\|_{L^2} &\lesssim (1+\tau)^{-\frac{np}{2(\mathtt{k}^+ -\delta)}(\frac{1}{m}-\frac{1}{2p})+ \frac{p\mathtt{k}^-}{2(\mathtt{k}^+ -\delta)}}\|(u,v)\|^p_{X(\tau)},
\end{align*}
where (\ref{GN1A1}) and (\ref{GN1A2}) hold for $p$. As a result, we derive
\begin{align*}
\big\|\partial_t^j |D|^{l\mathtt{k}^+} u^{nl}(t,\cdot)\big\|_{L^2} &\lesssim (1+t)^{-\frac{n}{2(\mathtt{k}^+ -\delta)}(\frac{1}{m}-\frac{1}{2})- \frac{l\mathtt{k}^+ + j(2\sigma- \mathtt{k}^-)- \mathtt{k}^-}{2(\mathtt{k}^+ -\delta)}}\|(u,v)\|^p_{X(t)} \int_0^{t/2}(1+\tau)^{-\frac{n}{2m(\mathtt{k}^+ -\delta)}(p-1)+ \frac{p\mathtt{k}^-}{2(\mathtt{k}^+ -\delta)}} d\tau\\
&\qquad + (1+t)^{-\frac{np}{2(\mathtt{k}^+ -\delta)}(\frac{1}{m}-\frac{1}{2p})+ \frac{p\mathtt{k}^-}{2(\mathtt{k}^+ -\delta)}}\|(u,v)\|^p_{X(t)} \int_{t/2}^t (1+t-\tau)^{- \frac{l\mathtt{k}^+ + j(2\sigma- \mathtt{k}^-)- \mathtt{k}^-}{2(\mathtt{k}^+ -\delta)}}d\tau.
\end{align*}
Here we used $(1+t-\tau) \approx (1+t) \text{ for any }\tau \in [0,t/2] \text{ and } (1+\tau) \approx (1+t) \text{ for any }\tau \in [t/2,t] $. Due to the condition $p \le 1+\frac{m(\mathtt{k}^+ +\sigma)}{n- m\mathtt{k}^-}$, the term $(1+\tau)^{-\frac{n}{2m(\mathtt{k}^+ -\delta)}(p-1)+ \frac{p\mathtt{k}^-}{2(\mathtt{k}^+ -\delta)}}$ is not integrable. Hence, we obtain
\begin{align*}
&(1+t)^{-\frac{n}{2(\mathtt{k}^+ -\delta)}(\frac{1}{m}-\frac{1}{2})- \frac{l\mathtt{k}^+ + j(2\sigma- \mathtt{k}^-)- \mathtt{k}^-}{2(\mathtt{k}^+ -\delta)}}\|(u,v)\|^p_{X(t)} \int_0^{t/2}(1+\tau)^{-\frac{n}{2m(\mathtt{k}^+ -\delta)}(p-1)+ \frac{p\mathtt{k}^-}{2(\mathtt{k}^+ -\delta)}} d\tau\\ 
&\qquad \lesssim (1+t)^{-\frac{n}{2(\mathtt{k}^+ -\delta)}(\frac{1}{m}-\frac{1}{2})- \frac{l\mathtt{k}^+ + j(2\sigma- \mathtt{k}^-)- \mathtt{k}^-}{2(\mathtt{k}^+ -\delta)}+\e(p)}\|(u,v)\|^p_{X(t)}.
\end{align*}
Moreover, we also notice that $\frac{l\mathtt{k}^+ + j(2\sigma- \mathtt{k}^-)- \mathtt{k}^-}{2(\mathtt{k}^+ -\delta)}< 1$ for $j,l=0,1$ and $(j,l) \neq (1,1)$. Consequently, we have
\begin{align*}
&(1+t)^{-\frac{np}{2(\mathtt{k}^+ -\delta)}(\frac{1}{m}-\frac{1}{2p})+ \frac{p\mathtt{k}^-}{2(\mathtt{k}^+ -\delta)}} \int_{t/2}^t (1+t-\tau)^{- \frac{l\mathtt{k}^+ + j(2\sigma- \mathtt{k}^-)- \mathtt{k}^-}{2(\mathtt{k}^+ -\delta)}}d\tau\\ 
&\qquad \lesssim (1+t)^{-\frac{n}{2(\mathtt{k}^+ -\delta)}(\frac{1}{m}-\frac{1}{2})- \frac{l\mathtt{k}^+ + j(2\sigma- \mathtt{k}^-)- \mathtt{k}^-}{2(\mathtt{k}^+ -\delta)}+\e(p)},
\end{align*}
provided that the condition (\ref{exponent1A}) holds for $q$. Finally, we derive the following estimate for $j,l=0,1$ and $(j,l) \neq (1,1)$:
$$\big\|\partial_t^j |D|^{l\mathtt{k}^+} u^{nl}(t,\cdot)\big\|_{L^2} \lesssim (1+t)^{-\frac{n}{2(\mathtt{k}^+ -\delta)}(\frac{1}{m}-\frac{1}{2})- \frac{l\mathtt{k}^+ + j(2\sigma- \mathtt{k}^-)- \mathtt{k}^-}{2(\mathtt{k}^+ -\delta)}+\e(p)}\|(u,v)\|^p_{X(t)}. $$
In the same way we arrive at the following estimate for $j,l=0,1$ and $(j,l) \neq (1,1)$:
$$\big\|\partial_t^j |D|^{l\mathtt{k}^+} v^{nl}(t,\cdot)\big\|_{L^2} \lesssim (1+t)^{-\frac{n}{2(\mathtt{k}^+ -\delta)}(\frac{1}{m}-\frac{1}{2})- \frac{l\mathtt{k}^+ + j(2\sigma- \mathtt{k}^-)- \mathtt{k}^-}{2(\mathtt{k}^+ -\delta)}}\|(u,v)\|^q_{X(t)}. $$
From the definition of the norm in $X(t)$, we may conclude immediately the inequality (\ref{pt4.31}). \medskip

\noindent Next, let us prove the estimate (\ref{pt4.4}). For two elements $(u,v)$ and $(\bar{u},\bar{v})$ from $X(t)$, we obtain
$$N(u,v)(t,x)- N(\bar{u},\bar{v})(t,x)= \big(u^{nl}(t,x)- \bar{u}^{nl}(t,x), v^{nl}(t,x)- \bar{v}^{nl}(t,x)\big). $$
Using again the $(L^m \cap L^2)- L^2$ estimates if $\tau \in [0,t/2]$ and the $L^2-L^2$ estimates if $\tau \in [t/2,t]$ from Corollary \ref{hq3.2}, we get the following estimate:
\begin{align*}
&\big\|\partial_t^j |D|^{l\mathtt{k}^+}\big((u^{nl}- \bar{u}^{nl})(t,\cdot)\big)\big\|_{L^2}\\
&\quad \lesssim \int_0^{t/2}(1+t-\tau)^{-\frac{n}{2(\mathtt{k}^+ -\delta)}(\frac{1}{m}-\frac{1}{2})- \frac{l\mathtt{k}^+ + j(2\sigma- \mathtt{k}^-)- \mathtt{k}^-}{2(\mathtt{k}^+ -\delta)}}\big\||v(\tau,\cdot)|^p- \bar{v}(\tau,\cdot)|^p\big\|_{L^m \cap L^2}d\tau\\
&\qquad + \int_{t/2}^t (1+t-\tau)^{- \frac{l\mathtt{k}^+ + j(2\sigma- \mathtt{k}^-)- \mathtt{k}^-}{2(\mathtt{k}^+ -\delta)}}\big\||v(\tau,\cdot)|^p- \bar{v}(\tau,\cdot)|^p\big\|_{L^2}d\tau,
\end{align*}
and
\begin{align*}
&\big\|\partial_t^j |D|^{l\mathtt{k}^+}\big((v^{nl}- \bar{v}^{nl})(t,\cdot)\big)\big\|_{L^2}\\
&\quad \lesssim \int_0^{t/2}(1+t-\tau)^{-\frac{n}{2(\mathtt{k}^+ -\delta)}(\frac{1}{m}-\frac{1}{2})- \frac{l\mathtt{k}^+ + j(2\sigma- \mathtt{k}^-)- \mathtt{k}^-}{2(\mathtt{k}^+ -\delta)}}\big\||u(\tau,\cdot)|^q- \bar{u}(\tau,\cdot)|^q\big\|_{L^m \cap L^2}d\tau\\
&\qquad + \int_{t/2}^t (1+t-\tau)^{- \frac{l\mathtt{k}^+ + j(2\sigma- \mathtt{k}^-)- \mathtt{k}^-}{2(\mathtt{k}^+ -\delta)}}\big\||u(\tau,\cdot)|^q- \bar{u}(\tau,\cdot)|^q\big\|_{L^2}d\tau.
\end{align*}
Employing H\"{o}lder's inequality gives
\begin{align*}
\big\||v(\tau,\cdot)|^p- |\bar{v}(\tau,\cdot)|^p\big\|_{L^2}& \lesssim \|v(\tau,\cdot)- \bar{v}(\tau,\cdot)\|_{L^{2p}} \big(\|v(\tau,\cdot)\|^{p-1}_{L^{2p}}+ \|\bar{v}(\tau,\cdot)\|^{p-1}_{L^{2p}}\big),\\
\big\||v(\tau,\cdot)|^p- |\bar{v}(\tau,\cdot)|^p\big\|_{L^m}& \lesssim \|v(\tau,\cdot)- \bar{v}(\tau,\cdot)\|_{L^{mp}} \big(\|v(\tau,\cdot)\|^{p-1}_{L^{mp}}+ \|\bar{v}(\tau,\cdot)\|^{p-1}_{L^{mp}}\big),\\
\big\||u(\tau,\cdot)|^q- |\bar{u}(\tau,\cdot)|^q\big\|_{L^2}& \lesssim \|u(\tau,\cdot)- \bar{u}(\tau,\cdot)\|_{L^{2q}} \big(\|u(\tau,\cdot)\|^{q-1}_{L^{2q}}+ \|\bar{u}(\tau,\cdot)\|^{q-1}_{L^{2q}}\big),\\
\big\||u(\tau,\cdot)|^q- |\bar{u}(\tau,\cdot)|^q\big\|_{L^m}& \lesssim \|u(\tau,\cdot)- \bar{u}(\tau,\cdot)\|_{L^{mq}} \big(\|u(\tau,\cdot)\|^{q-1}_{L^{mq}}+ \|\bar{u}(\tau,\cdot)\|^{q-1}_{L^{mq}}\big).
\end{align*}
Similarly to the proof of (\ref{pt4.3}), we apply the fractional Gagliardo-Nirenberg inequality from Proposition \ref{fractionalGagliardoNirenberg} to the terms
$$ \|v(\tau,\cdot)-\bar{v}(\tau,\cdot)\|_{L^{\eta_1}},\,\, \|u(\tau,\cdot)-\bar{u}(\tau,\cdot)\|_{L^{\eta_2}},\,\, \|v(\tau,\cdot)\|_{L^{\eta_1}},\,\, \|\bar{v}(\tau,\cdot)\|_{L^{\eta_1}},\,\, \|u(\tau,\cdot)\|_{L^{\eta_2}},\,\, \|\bar{u}(\tau,\cdot)\|_{L^{\eta_2}} $$
with $\eta_1=2p$ or $\eta_1=mp$, and $\eta_2=2q$ or $\eta_2=mq$ to conclude the inequality (\ref{pt4.4}). Summarizing, Theorem 1-A is proved.

\begin{nx} \label{remark3.2}
\fontshape{n}
\selectfont
The proof of Theorem 1-B is similar to the proof of Theorem 1-A. Here we notice that due to the condition (\ref{exponent1B}), the terms $(1+\tau)^{-\frac{n}{2m(\mathtt{k}^+ -\delta)}(p-1)+ \frac{p\mathtt{k}^-}{2(\mathtt{k}^+ -\delta)}}$ and $(1+\tau)^{-\frac{n}{2m(\mathtt{k}^+ -\delta)}(q-1)+ \frac{q\mathtt{k}^-}{2(\mathtt{k}^+ -\delta)}}$ are integrable. Then, repeating some of the arguments as we did in the proof of Theorem 1-A we may complete the proof of Theorem 1-B.\medskip
\end{nx}

\subsection{Proof of Theorem 2-A: $0< s_1 \le s_2 < \mathtt{k}^+$}
We introduce the solution space
$$X(t):= \big(C([0,t],H^{s_1,q})\big) \times \big(C([0,t],H^{s_2,q})\big), $$
where the weights are modified in the following way:
\begin{align*}
&f_{1}(\tau)= (1+\tau)^{-\frac{n}{2(\mathtt{k}^+ -\delta)}(\frac{1}{m}-\frac{1}{2})+ \frac{\mathtt{k}^-}{2(\mathtt{k}^+ -\delta)}+ \e(p)},\,\,\, f_{1,s_1}(\tau)=(1+\tau)^{-\frac{n}{2(\mathtt{k}^+ -\delta)}(\frac{1}{m}-\frac{1}{2})- \frac{s_1 - \mathtt{k}^-}{2(\mathtt{k}^+ -\delta)}+ \e(p)}, \\
&f_{2}(\tau)= g_{2}(\tau)= f_{2,s_1}(\tau)= g_{2,s_2}(\tau) \equiv 0.
\end{align*}
In order to prove the two inequalities (\ref{pt4.31}) and (\ref{pt4.4}), we use the $(L^m \cap L^2)- L^2$ estimates if $\tau \in [0,t/2]$ and the $L^2-L^2$ estimates if $\tau \in [t/2,t]$ from Corollary \ref{hq3.2}. Hence, we get the following estimates for $l=0,1$:
\begin{align*}
\big\||D|^{ls_1} u^{nl}(t,\cdot)\big\|_{L^2} &\lesssim \int_0^{t/2}(1+t-\tau)^{-\frac{n}{2(\mathtt{k}^+ -\delta)}(\frac{1}{m}-\frac{1}{2})- \frac{ls_1 - \mathtt{k}^-}{2(\mathtt{k}^+ -\delta)}}\big\||v(\tau,\cdot)|^p\big\|_{L^m \cap L^2}d\tau\\
&\qquad + \int_{t/2}^t (1+t-\tau)^{- \frac{ls_1 - \mathtt{k}^-}{2(\mathtt{k}^+ -\delta)}}\big\||v(\tau,\cdot)|^p\big\|_{L^2}d\tau,
\end{align*}
 and
\begin{align*}
\big\||D|^{ls_1}\big((u^{nl}- \bar{u}^{nl})(t,\cdot)\big)\big\|_{L^2} &\lesssim \int_0^{t/2}(1+t-\tau)^{-\frac{n}{2(\mathtt{k}^+ -\delta)}(\frac{1}{m}-\frac{1}{2})- \frac{ls_1 - \mathtt{k}^-}{2(\mathtt{k}^+ -\delta)}}\big\||v(\tau,\cdot)|^p- \bar{v}(\tau,\cdot)|^p\big\|_{L^m \cap L^2}d\tau\\
&\qquad + \int_{t/2}^t (1+t-\tau)^{- \frac{ls_1 - \mathtt{k}^-}{2(\mathtt{k}^+ -\delta)}}\big\||v(\tau,\cdot)|^p- \bar{v}(\tau,\cdot)|^p\big\|_{L^2}d\tau.
\end{align*}
Similar to the treatment of Theorem 1-A, we arrive at the following estimates for $l=0,1$:
\begin{align*}
&\big\||D|^{ls_1}u^{nl}(t,\cdot)\big\|_{L^2} \lesssim (1+t)^{-\frac{n}{2(\mathtt{k}^+ -\delta)}(\frac{1}{m}-\frac{1}{2})- \frac{ls_1 - \mathtt{k}^-}{2(\mathtt{k}^+ -\delta)}+\e(p)}\|(u,v)\|^p_{X(t)},\\
&\big\||D|^{ls_1}\big((u^{nl}- \bar{u}^{nl})(t,\cdot)\big)\big\|_{L^2}\\
&\qquad \lesssim (1+t)^{-\frac{n}{2(\mathtt{k}^+ -\delta)}(\frac{1}{m}-\frac{1}{2})- \frac{ls_1 - \mathtt{k}^-}{2(\mathtt{k}^+ -\delta)}+\e(p)}\|(u,v)- (\bar{u},\bar{v})\|_{X(t)} \big(\|(u,v)\|^{p-1}_{X(t)}+ \|(\bar{u},\bar{v})\|^{p-1}_{X(t)}\big),
\end{align*}
provided that the conditions (\ref{GN2A1}) to (\ref{exponent2A}) hold. Analogously, we obtain the following estimates for $l=0,1$:
\begin{align*}
\big\||D|^{ls_2}v^{nl}(t,\cdot)\big\|_{L^2} &\lesssim (1+t)^{-\frac{n}{2(\mathtt{k}^+ -\delta)}(\frac{1}{m}-\frac{1}{2})- \frac{ls_2 - \mathtt{k}^-}{2(\mathtt{k}^+ -\delta)}}\|(u,v)\|^q_{X(t)},\\
\big\||D|^{ls_2}\big((v^{nl}- \bar{v}^{nl})(t,\cdot)\big)\big\|_{L^2} &\lesssim (1+t)^{-\frac{n}{2(\mathtt{k}^+ -\delta)}(\frac{1}{m}-\frac{1}{2})- \frac{ls_2 - \mathtt{k}^-}{2(\mathtt{k}^+ -\delta)}}\|(u,v)- (\bar{u},\bar{v})\|_{X(t)} \big(\|(u,v)\|^{q-1}_{X(t)}+ \|(\bar{u},\bar{v})\|^{q-1}_{X(t)}\big),
\end{align*}
From the definition of the norm in $X(t)$ we may conclude immediately the inequalities (\ref{pt4.31}) and (\ref{pt4.4}). Summarizing, Theorem 2-A is proved.

\begin{nx}
\fontshape{n}
\selectfont
Like in Remark \ref{remark3.2}, the proof of Theorem 2-B is similar to the proof of Theorem 2-A. Then, repeating some of the arguments as we did in the proof of Theorem 2-A we may complete the proof Theorem 2-B.\medskip
\end{nx}

\subsection{Proof of Theorem $3$: $\mathtt{k}^+ < s_1 \le s_2 \le \frac{n}{2}+ \mathtt{k}^+$}
We introduce the solution space
$$X(t):= \Big(C([0,t],H^{s_1}) \cap C^1([0,t],H^{s_1- \mathtt{k}^+})\Big) \times \Big(C([0,t],H^{s_2}) \cap C^1([0,t],H^{s_2- \mathtt{k}^+})\Big). $$
\noindent First, let us prove the inequality (\ref{pt4.31}). In the first step, it is necessary to estimate the following norms:
$$\|u^{nl}(t,\cdot)\|_{L^2},\,\, \|u_t^{nl}(t,\cdot)\|_{L^2},\,\, \big\||D|^{s_1}u^{nl}(t,\cdot)\big\|_{L^2},\,\, \big\||D|^{s_1- \mathtt{k}^+}u_t^{nl}(t,\cdot)\big\|_{L^2}. $$
Similar to the treatment of Theorem 1-A we arrive at the following estimates for $j=0,1$:
\begin{equation} \label{t4A1}
\big\|\partial_t^j u^{nl}(t,\cdot)\big\|_{L^2} \lesssim (1+t)^{-\frac{n}{2(\mathtt{k}^+ -\delta)}(\frac{1}{m}-\frac{1}{2})- \frac{j(2\sigma- \mathtt{k}^-)- \mathtt{k}^-}{2(\mathtt{k}^+ -\delta)}}\|(u,v)\|^p_{X(t)},
\end{equation}
where condition $p > 1+ \frac{m(\mathtt{k}^+ +\sigma)}{n- m\mathtt{k}^-}$ holds and
\begin{equation*}
p\in \Big[\frac{2}{m},\ity \Big)  \text{ if } n\le 2s_1, \text{ or }p \in \Big[\f{2}{m}, \f{n}{n-2s_1}\Big] \text{ if } n \in \Big(2s_1, \f{4s_1}{2-m}\Big].
\end{equation*}
\noindent Now, let us turn to control the norm $\big\||D|^{s_1} u^{nl}(t,\cdot)\big\|_{L^2}$. We use the $(L^m \cap L^2)- L^2$ estimates if $\tau \in [0,t/2]$ and the $L^2-L^2$ estimates if $\tau \in [t/2,t]$ from Corollary \ref{hq3.2} to get
\begin{align*}
\big\||D|^{s_1} u^{nl}(t,\cdot)\big\|_{L^2} &\lesssim \int_0^{t/2}(1+t-\tau)^{-\frac{n}{2(\mathtt{k}^+ -\delta)}(\frac{1}{m}-\frac{1}{2})- \frac{s_1 - \mathtt{k}^-}{2(\mathtt{k}^+ -\delta)}}\big\||v(\tau,\cdot)|^p\big\|_{L^m \cap L^2 \cap \dot{H}^{s_1- \mathtt{k}^+}}d\tau\\
&\qquad + \int_{t/2}^t (1+t-\tau)^{- \frac{s_1 - \mathtt{k}^-}{2(\mathtt{k}^+ -\delta)}}\big\||v(\tau,\cdot)|^p\big\|_{L^2 \cap \dot{H}^{s_1- \mathtt{k}^+}}d\tau,
\end{align*}
The integrals with $\big\||v(\tau,\cdot)|^p\big\|_{L^m \cap L^2}$ and $\big\||v(\tau,\cdot)|^p\big\|_{L^ 2}$ will be handled as we did to obtain (\ref{t4A1}). In order to control the integral with $\big\||v(\tau,\cdot)|^p\big\|_{\dot{H}^{s_1- \mathtt{k}^+}}$, we shall apply Proposition \ref{Propfractionalchainrulegeneral} for the fractional chain rule with $p> \lceil s_1- \mathtt{k}^+ \rceil$ and Proposition \ref{fractionalGagliardoNirenberg} for the fractional Gagliardo-Nirenberg inequality. Consequently, we derive
\begin{align*}
&\big\||v(\tau,\cdot)|^p\big\|_{\dot{H}^{s_1- \mathtt{k}^+}} \lesssim \|v(\tau,\cdot)\|^{p-1}_{L^{q_1}}\,\,\big\||D|^{s_1- \mathtt{k}^+}v(\tau,\cdot)\big\|_{L^{q_2}}\\
&\qquad \lesssim \|v(\tau,\cdot)\|^{(p-1)(1-\theta_{q_1})}_{L^2}\,\,\big\||D|^{s_2} v(\tau,\cdot)\big\|^{(p-1)\theta_{q_1}}_{L^2}\,\,\|u(\tau,\cdot)\|^{1-\theta_{q_2}}_{L^2}\,\,\big\||D|^{s_2} v(\tau,\cdot)\big\|^{\theta_{q_2}}_{L^2}\\
&\qquad \lesssim (1+\tau)^{-\frac{np}{2(\mathtt{k}^+ -\delta)}(\frac{1}{m}-\frac{1}{2p})+ \frac{p\mathtt{k}^-}{2(\mathtt{k}^+ -\delta)}- \frac{s_1 - \mathtt{k}^+}{2(\mathtt{k}^+ -\delta)}}\|(u,v)\|^p_{X(\tau)},
\end{align*}
where
$$\frac{p-1}{q_1}+\frac{1}{q_2}= \frac{1}{2},\,\, \theta_{q_1}= \frac{n}{s_2}\Big(\frac{1}{2}-\frac{1}{q_1}\Big) \in [0,1],\,\, \theta_{q_2}= \frac{n}{s_2}\Big(\frac{1}{2}-\frac{1}{q_2}+\frac{s_1- \mathtt{k}^+}{n}\Big) \in \Big[\frac{s_1- \mathtt{k}^+}{s_2},1\Big]. $$
From the above conditions we deduce the following restriction for $p$:
\begin{equation*}
1<p\le 1+\frac{2\mathtt{k}^+}{n-2s_2} \text{ if } n>2s_2, \text{ or } p>1 \text{ if } n \le 2s_2.
\end{equation*}
Hence, we arrive at
\begin{equation}  \label{t4A2}
\big\||D|^{s_1}u^{nl}(t,\cdot)\big\|_{L^2} \lesssim (1+t)^{-\frac{n}{2(\mathtt{k}^+ -\delta)}(\frac{1}{m}-\frac{1}{2})- \frac{s_1 - \mathtt{k}^-}{2(\mathtt{k}^+ -\delta)}}\|(u,v)\|^p_{X(t)},
\end{equation}
where the condition $p \ge 1+ \frac{ms_1}{n- m\mathtt{k}^-}$ holds. In the analogous way we also have
\begin{equation}  \label{t4A3}
\big\||D|^{s_1- \mathtt{k}^+}u_t^{nl}(t,\cdot)\big\|_{L^2} \lesssim (1+t)^{-\frac{n}{2(\mathtt{k}^+ -\delta)}(\frac{1}{m}-\frac{1}{2})- \frac{s_1+ \mathtt{k}^+ - 4\delta}{2(\mathtt{k}^+ -\delta)}}\|(u,v)\|^p_{X(t)}.
\end{equation}
Similarly, with the assumption $s_2- \mathtt{k}^+ < s_1$ we obtain the following estimates for $j=0,1$ in the second step:
\begin{align}
\big\|\partial_t^j v^{nl}(t,\cdot)\big\|_{L^2} &\lesssim (1+t)^{-\frac{n}{2(\mathtt{k}^+ -\delta)}(\frac{1}{m}-\frac{1}{2})- \frac{j(2\sigma- \mathtt{k}^-)- \mathtt{k}^-}{2(\mathtt{k}^+ -\delta)}}\|(u,v)\|^q_{X(t)}, \label{t4A4} \\
\big\||D|^{s_2}v^{nl}(t,\cdot)\big\|_{L^2} &\lesssim (1+t)^{-\frac{n}{2(\mathtt{k}^+ -\delta)}(\frac{1}{m}-\frac{1}{2})- \frac{s_2 - \mathtt{k}^-}{2(\mathtt{k}^+ -\delta)}}\|(u,v)\|^q_{X(t)}, \label{t4A5} \\
\big\||D|^{s_2- \mathtt{k}^+}v_t^{nl}(t,\cdot)\big\|_{L^2} &\lesssim (1+t)^{-\frac{n}{2(\mathtt{k}^+ -\delta)}(\frac{1}{m}-\frac{1}{2})- \frac{s_2+ \mathtt{k}^+ - 4\delta}{2(\mathtt{k}^+ -\delta)}}\|(u,v)\|^q_{X(t)}. \label{t4A6}
\end{align}
Here the condition $q > 1+ \frac{m(\mathtt{k}^+ +\sigma)}{n- m\mathtt{k}^-}$ is fulfilled and the conditions (\ref{GN3A1}) to (\ref{exponent3A}) hold for $q$. Summarizing, from (\ref{t4A1}) to (\ref{t4A6}) and the definition of the norm in $X(t)$ we may conclude immediately the inequality (\ref{pt4.31}).\medskip

\noindent Next, let us prove the inequality (\ref{pt4.4}). We can follow, on the one hand, the proof of Theorem 1-A. On the other hand, we need to control the norm $\big\||v(\tau,\cdot)|^p-|\bar{v}(\tau,\cdot)|^p\big\|_{\dot{H}^{s_1- \mathtt{k}^+}}$. The integral representation
$$ |v(\tau,x)|^p-|\bar{v}(\tau,x)|^p=p\int_0^1 \big(v(\tau,x)-\bar{v}(\tau,x)\big)G\big(\omega v(\tau,x)+(1-\omega)\bar{v}(\tau,x)\big)d\omega, $$
where $G(v)=v|v|^{p-2}$ gives
$$\big\||v(\tau,\cdot)|^p-|\bar{v}(\tau,\cdot)|^p\big\|_{\dot{H}^{s_1- \mathtt{k}^+}} \lesssim \int_0^1 \Big\||D|^{s_1- \mathtt{k}^+}\Big(\big(v(\tau,\cdot)-\bar{v}(\tau,\cdot)\big)G\big(\omega v(\tau,\cdot)+(1-\omega)\bar{v}(\tau,\cdot)\big)\Big)\Big\|_{L^2}d\omega. $$
Employing the fractional Leibniz formula from Proposition \ref{fractionalLeibniz} we arrive at
\begin{align*}
\big\||v(\tau,\cdot)|^p-|\bar{v}(\tau,\cdot)|^p\big\|_{\dot{H}^{s_1- \mathtt{k}^+}} &\lesssim \big\||D|^{s_1- \mathtt{k}^+}\big(v(\tau,\cdot)-\bar{v}(\tau,\cdot)\big)\big\|_{L^{r_1}} \int_0^1 \big\|G\big(\omega v(\tau,\cdot)+(1-\omega)\bar{v}(\tau,\cdot)\big)\big\|_{L^{r_2}}d\omega\\
&\quad + \|v(\tau,\cdot)-\bar{v}(\tau,\cdot)\|_{L^{r_3}}  \int_0^1 \big\||D|^{s_1- \mathtt{k}^+}G\big(\omega v(\tau,\cdot)+(1-\omega)\bar{v}(\tau,\cdot)\big)\big\|_{L^{r_4}}d\omega\\
&\lesssim \big\||D|^{s_1- \mathtt{k}^+}\big(v(\tau,\cdot)-\bar{v}(\tau,\cdot)\big)\big\|_{L^{r_1}} \Big(\|v(\tau,\cdot)\|^{p-1}_{L^{r_2 (p-1)}}+ \|\bar{v}(\tau,\cdot)\|^{p-1}_{L^{r_2 (p-1)}}\Big)\\
&\quad + \|v(\tau,\cdot)-\bar{v}(\tau,\cdot)\|_{L^{r_3}}  \int_0^1 \big\||D|^{s_1- \mathtt{k}^+}G\big(\omega v(\tau,\cdot)+(1-\omega)\bar{v}(\tau,\cdot)\big)\big\|_{L^{r_4}}d\omega,
\end{align*}
where
$$\frac{1}{r_1}+\frac{1}{r_2}= \frac{1}{r_3}+\frac{1}{r_4}= \frac{1}{2}.$$
Applying the fractional Gargliardo-Nirenberg inequality from Proposition \ref{fractionalGagliardoNirenberg} we derive
\begin{align*}
\big\||D|^{s_1- \mathtt{k}^+}\big(v(\tau,\cdot)-\bar{v}(\tau,\cdot)\big)\big\|_{L^{r_1}}&\lesssim \|v(\tau,\cdot)- \bar{v}(\tau,\cdot)\|^{\theta_1}_{\dot{H}^{s_2}}\,\,\|v(\tau,\cdot)-\bar{v}(\tau,\cdot)\|^{1-\theta_1}_{L^2},\\
\|v(\tau,\cdot)\|_{L^{r_2 (p-1)}}&\lesssim \|v(\tau,\cdot)\|^{\theta_2}_{\dot{H}^{s_2}}\,\,\|v(\tau,\cdot)\|^{1-\theta_2}_{L^2},\\
\|v(\tau,\cdot)- \bar{v}(\tau,\cdot)\|_{L^{r_3}}&\lesssim \|v(\tau,\cdot)- \bar{v}(\tau,\cdot)\|^{\theta_3}_{\dot{H}^{s_2}}\,\,\|v(\tau,\cdot)- \bar{v}(\tau,\cdot)\|^{1-\theta_3}_{L^2},
\end{align*}
where
$$\theta_1= \frac{n}{s_2}\Big(\frac{1}{2}-\frac{1}{r_1}+\frac{s_1- \mathtt{k}^+}{n}\Big) \in \Big[\frac{s_1- \mathtt{k}^+}{s_2},1\Big],\,\, \theta_2= \frac{n}{s_2}\Big(\frac{1}{2}-\frac{1}{r_2(p-1)}\Big) \in [0,1],\,\, \theta_3= \frac{n}{s_2}\Big(\frac{1}{2}-\frac{1}{r_3}\Big) \in [0,1]. $$
Since $\omega \in [0,1]$ is a parameter, applying again the fractional chain rule from Proposition \ref{Propfractionalchainrulegeneral} with $p >1+ \lceil s_1- \mathtt{k}^+ \rceil$ and the fractional Gagliardo-Nirenberg inequality from Proposition \ref{fractionalGagliardoNirenberg} leads to
\begin{align*}
&\big\||D|^{s_1- \mathtt{k}^+}G\big(\omega v(\tau,\cdot)+(1-\omega)\bar{v}(\tau,\cdot)\big)\big\|_{L^{r_4}}\\
&\qquad \lesssim \|\omega v(\tau,\cdot)+(1-\omega)\bar{v}(\tau,\cdot)\|^{p-2}_{L^{r_5}}\,\, \big\||D|^{s_1- \mathtt{k}^+}\big(\omega v(\tau,\cdot)+(1-\omega)\bar{v}(\tau,\cdot)\big)\big\|_{L^{r_6}}\\
&\qquad \lesssim \|\omega v(\tau,\cdot)+(1-\omega)\bar{v}(\tau,\cdot)\|^{(p-2)\theta_5+\theta_6}_{\dot{H}^{s_2}}\,\, \|\omega v(\tau,\cdot)+(1-\omega)\bar{v}(\tau,\cdot)\|^{(p-2)(1-\theta_5)+1-\theta_6}_{L^2},
\end{align*}
where
$$\frac{p-2}{r_5}+\frac{1}{r_6}= \frac{1}{r_4},\,\, \theta_5= \frac{n}{s_2}\Big(\frac{1}{2}-\frac{1}{r_5}\Big) \in [0,1],\,\, \theta_6= \frac{n}{s_2}\Big(\frac{1}{2}-\frac{1}{r_6}+\frac{s_1- \mathtt{k}^+}{n}\Big) \in \Big[\frac{s_1- \mathtt{k}^+}{s_2},1\Big]. $$
Consequently, we obtain
\begin{align*}
&\int_0^1 \big\||D|^{s_1- \mathtt{k}^+}G\big(\omega v(\tau,\cdot)+(1-\omega)\bar{v}(\tau,\cdot)\big)\big\|_{L^{r_4}}d\omega\\
&\quad \lesssim \big(\|v(\tau,\cdot)\|_{\dot{H}^{s_2}}+\|\bar{v}(\tau,\cdot)\|_{\dot{H}^{s_2}}\big)^{(p-2)\theta_5+\theta_6}\, \big(\|v(\tau,\cdot)\|_{L^2}+\|\bar{v}(\tau,\cdot)\|_{L^2} \big)^{(p-2)(1-\theta_5)+1-\theta_6}.
\end{align*}
Hence, we drived the following estimate:
\begin{align*}
&\big\||v(\tau,\cdot)|^p-|\bar{v}(\tau,\cdot)|^p\big\|_{\dot{H}^{s_1- \mathtt{k}^+}}\\ 
&\qquad \lesssim (1+\tau)^{-\frac{np}{2(\mathtt{k}^+ -\delta)}(\frac{1}{m}-\frac{1}{2p})+ \frac{p\mathtt{k}^-}{2(\mathtt{k}^+ -\delta)}- \frac{s_1 - \mathtt{k}^+}{2(\mathtt{k}^+ -\delta)}}\|(u,v)- (\bar{u},\bar{v})\|_{X(t)} \big(\|(u,v)\|^{p-1}_{X(t)}+ \|(\bar{u},\bar{v})\|^{p-1}_{X(t)}\big),
\end{align*}
where we note that
$$\theta_1+ (p-1)\theta_2= \theta_3+ (p-2)\theta_5+ \theta_6= \frac{n}{s_2}\Big(\frac{p-1}{2}+\frac{s_1- \mathtt{k}^+}{n}\Big). $$
Finally, we have proved that
\begin{align*}
&\big\||D|^{s_1}\big((u^{nl}- \bar{u}^{nl})(t,\cdot)\big)\big\|_{L^2} \\ 
&\qquad \lesssim (1+t)^{-\frac{n}{2(\mathtt{k}^+ -\delta)}(\frac{1}{m}-\frac{1}{2})- \frac{s_1 - \mathtt{k}^-}{2(\mathtt{k}^+ -\delta)}}\|(u,v)- (\bar{u},\bar{v})\|_{X(t)} \big(\|(u,v)\|^{p-1}_{X(t)}+ \|(\bar{u},\bar{v})\|^{p-1}_{X(t)}\big),
\end{align*}
and
\begin{align*}
&\big\||D|^{s_1- \mathtt{k}^+}\big((u_t^{nl}- \bar{u_t}^{nl})(t,\cdot)\big)\big\|_{L^2} \\
&\qquad \lesssim (1+t)^{-\frac{n}{2(\mathtt{k}^+ -\delta)}(\frac{1}{m}-\frac{1}{2})- \frac{s_1+ \mathtt{k}^+ -  4\delta}{2(\mathtt{k}^+ -\delta)}}\|(u,v)- (\bar{u},\bar{v})\|_{X(t)} \big(\|(u,v)\|^{p-1}_{X(t)}+ \|(\bar{u},\bar{v})\|^{p-1}_{X(t)}\big).
\end{align*}
In the same way we arrive at
\begin{align*}
&\big\||D|^{s_2}\big((v^{nl}- \bar{v}^{nl})(t,\cdot)\big)\big\|_{L^2} \\ 
&\qquad \lesssim (1+t)^{-\frac{n}{2(\mathtt{k}^+ -\delta)}(\frac{1}{m}-\frac{1}{2})- \frac{s_2 - \mathtt{k}^-}{2(\mathtt{k}^+ -\delta)}}\|(u,v)- (\bar{u},\bar{v})\|_{X(t)} \big(\|(u,v)\|^{q-1}_{X(t)}+ \|(\bar{u},\bar{v})\|^{q-1}_{X(t)}\big),
\end{align*}
and
\begin{align*}
&\big\||D|^{s_2- \mathtt{k}^+}\big((v_t^{nl}- \bar{v_t}^{nl})(t,\cdot)\big)\big\|_{L^2} \\
&\qquad \lesssim (1+t)^{-\frac{n}{2(\mathtt{k}^+ -\delta)}(\frac{1}{m}-\frac{1}{2})- \frac{s_2+ \mathtt{k}^+ - 4\delta}{2(\mathtt{k}^+ -\delta)}}\|(u,v)- (\bar{u},\bar{v})\|_{X(t)} \big(\|(u,v)\|^{q-1}_{X(t)}+ \|(\bar{u},\bar{v})\|^{q-1}_{X(t)}\big),
\end{align*}
provided that the condition $q > 1+ \lceil s_2- \mathtt{k}^+ \rceil$ is satisfied. From the definition of the norm in $X(t)$ we may conclude immediately the inequality (\ref{pt4.4}). Summarizing, the proof of Theorem $3$ is completed.

\begin{nx}
\fontshape{n}
\selectfont
It is clear to explain the possibility to choose suitable parameters $q_1$, $q_2$, $r_1,\cdots, r_6$ and $\theta_1,\cdots,\theta_6$ as required in the proof to Theorem $3$. Following the explanations as we did in Remark $4.2$ in \cite{DaoReissig} we may conclude that the following conditions are sufficient to guarantee the existence of all these parameters satisfying the required conditions:
$$ \begin{cases}
2\le p \le 1+\frac{2\mathtt{k}^+}{n-2s_2} \text{ if } n>2s_2, \text{ or }p \ge 2 \text{ if }n \le 2s_2, \\ 
2\le q \le 1+\frac{2\mathtt{k}^+}{n-2s_1} \text{ if } n>2s_1, \text{ or }q \ge 2 \text{ if }n \le 2s_1.
\end{cases} $$
\end{nx}

\subsection{Proof of Theorem $4$: $s_2 \ge s_1> \frac{n}{2}+ \mathtt{k}^+$}
We introduce both spaces for the data and the solutions as in Theorem $3$. We can repeat exactly, on the one hand, the estimates of the following terms:
$$|v(\tau,\cdot)|^p,\,\, |v(\tau,\cdot)|^p-|\bar{v}(\tau,\cdot)|^p,\,\, |u(\tau,\cdot)|^q,\,\, |u(\tau,\cdot)|^q- |\bar{u}(\tau,\cdot)|^q $$
in $L^m$ and $L^2$ as we did in the proof to Theorem $3$. On the other hand, let us estimate the two first terms in $\dot{H}^{s_1- \mathtt{k}^+}$ and the two remaining terms in $\dot{H}^{s_2- \mathtt{k}^+}$ by using the fractional powers rule and the fractional Sobolev embedding.\medskip

In the first step, let us control the norm $\big\||v(\tau,\cdot)|^p\big\|_{\dot{H}^{s_1- \mathtt{k}^+}}$. We shall apply Corollary \ref{Corfractionalhomogeneous} for fractional powers with $s_1- \mathtt{k}^+ \in \big(\frac{n}{2},p\big)$ and Lemma \ref{LemmaEmbedding} with a suitable $s_1^* <\frac{n}{2}$ to get
$$ \big\||v(\tau,\cdot)|^p\big\|_{\dot{H}^{s_1- \mathtt{k}^+}}\lesssim \|v(\tau,\cdot)\|_{\dot{H}^{s_1- \mathtt{k}^+}}\|v(\tau,\cdot)\|^{p-1}_{L^\ity} \lesssim \|v(\tau,\cdot)\|_{\dot{H}^{s_1- \mathtt{k}^+}}\big(\|v(\tau,\cdot)\|_{\dot{H}^{s_1^*}}+ \|v(\tau,\cdot)\|_{\dot{H}^{s_1- \mathtt{k}^+}}\big)^{p-1}. $$
Using the fractional Gagliardo-Nirenberg inequality from Proposition \ref{fractionalGagliardoNirenberg} leads to
\begin{align*}
\|v(\tau,\cdot)\|_{\dot{H}^{s_1- \mathtt{k}^+}} &\lesssim \|v(\tau,\cdot)\|^{1-\theta_1}_{L^2}\big\||D|^{s_2} v(\tau,\cdot)\big\|^{\theta_1}_{L^2} \lesssim (1+\tau)^{-\frac{n}{2(\mathtt{k}^+ -\delta)}(\frac{1}{m}-\frac{1}{2})- \frac{s_1 - \mathtt{k}^+ -\mathtt{k}^-}{2(\mathtt{k}^+ -\delta)}}\|(u,v)\|^p_{X(\tau)}, \\
\|v(\tau,\cdot)\|_{\dot{H}^{s_1^*}} &\lesssim \|v(\tau,\cdot)\|^{1-\theta_2}_{L^2}\big\||D|^{s_2} v(\tau,\cdot)\big\|^{\theta_2}_{L^2} \lesssim (1+\tau)^{-\frac{n}{2(\mathtt{k}^+ -\delta)}(\frac{1}{m}-\frac{1}{2})- \frac{s_1^* - \mathtt{k}^-}{2(\mathtt{k}^+ -\delta)}}\|(u,v)\|^p_{X(\tau)},
\end{align*}
where $\theta_1= \frac{s_1- \mathtt{k}^+}{s_2}$ and $\theta_2= \frac{s_1^*}{s_2}$. Consequently, we have
\begin{align*}
&\big\||v(\tau,\cdot)|^p\big\|_{\dot{H}^{s_1- \mathtt{k}^+}}\lesssim (1+\tau)^{-\frac{np}{2(\mathtt{k}^+ -\delta)}(\frac{1}{m}-\frac{1}{2})+\frac{p\mathtt{k}^-}{2(\mathtt{k}^+ -\delta)}- \frac{s_1 - \mathtt{k}^+}{2(\mathtt{k}^+ -\delta)}- (p-1)\frac{s_1^*}{2(\mathtt{k}^+ -\delta)}}\|(u,v)\|^p_{X(\tau)} \\ 
&\qquad \lesssim (1+\tau)^{-\frac{np}{2(\mathtt{k}^+ -\delta)}(\frac{1}{m}-\frac{1}{2p})+ \frac{p\mathtt{k}^-}{2(\mathtt{k}^+ -\delta)}}\|(u,v)\|^p_{X(\tau)},
\end{align*}
if we choose $s_1^*= \frac{n}{2}-\e$ with a sufficiently small positive number $\e$. In the same way we arrive at the following estimate:
$$\big\||u(\tau,\cdot)|^q\big\|_{\dot{H}^{s_2- \mathtt{k}^+}}\lesssim (1+\tau)^{-\frac{nq}{2(\mathtt{k}^+ -\delta)}(\frac{1}{m}-\frac{1}{2q})+ \frac{q\mathtt{k}^-}{2(\mathtt{k}^+ -\delta)}}\|(u,v)\|^q_{X(\tau)}, $$
provided that the condition $q >s_2- \mathtt{k}^+$ is fulfilled.\medskip

\noindent Next, let us control the norm $\big\||v(\tau,\cdot)|^p-|\bar{v}(\tau,\cdot)|^p\big\|_{\dot{H}^{s_1- \mathtt{k}^+}}$. Then, repeating the proof of Theorem $3$ and using the analogous arguments as in the first step we get
\begin{align*}
&\big\||v(\tau,\cdot)|^p-|\bar{v}(\tau,\cdot)|^p\big\|_{\dot{H}^{s_1- \mathtt{k}^+}} \\ 
&\qquad \lesssim (1+\tau)^{-\frac{np}{2(\mathtt{k}^+ -\delta)}(\frac{1}{m}-\frac{1}{2p})+ \frac{p\mathtt{k}^-}{2(\mathtt{k}^+ -\delta)}} \|(u,v)- (\bar{u},\bar{v})\|_{X(t)}\big( \|(u,v)\|^{p-1}_{X(t)}+ \|(\bar{u},\bar{v})\|^{p-1}_{X(t)} \big),
\end{align*}
and
\begin{align*}
&\big\||u(\tau,\cdot)|^q-|\bar{u}(\tau,\cdot)|^q\big\|_{\dot{H}^{s_2- \mathtt{k}^+}} \\ 
&\qquad \lesssim (1+\tau)^{-\frac{nq}{2(\mathtt{k}^+ -\delta)}(\frac{1}{m}-\frac{1}{2q})+ \frac{q\mathtt{k}^-}{2(\mathtt{k}^+ -\delta)}} \|(u,v)- (\bar{u},\bar{v})\|_{X(t)}\big( \|(u,v)\|^{q-1}_{X(t)}+ \|(\bar{u},\bar{v})\|^{q-1}_{X(t)} \big),
\end{align*}
provided that the conditions $p,\,q>2$, $p> 1+s_1- \mathtt{k}^+$ and $q> 1+s_2- \mathtt{k}^+$ hold. Summarizing, Theorem $4$ is proved.

\subsection{Proof of Theorem $5$: $s_1=s_2=s > \frac{n}{2}+ \mathtt{k}^+$}
We introduce the solution space
$$X(t):= \Big(C([0,t],H^{s}) \cap C^1([0,t],H^{s- \mathtt{k}^+})\Big)^2. $$
First, let us prove the inequality (\ref{pt4.31}). In the first step, to deal with $\partial_t^j u^{nl}$ for $j=0,1$ we apply the $(L^m \cap L^2)- L^2$ estimates if $\tau \in [0,t/2]$ and the $L^2-L^2$ estimates if $\tau \in [t/2,t]$ from Corollary \ref{hq3.2}. As a result, we obtain the following estimate for $j=0,1$:
\begin{align*}
\big\|\partial_t^j u^{nl}(t,\cdot)\big\|_{L^2} &\lesssim \int_0^{t/2}(1+t-\tau)^{-\frac{n}{2(\mathtt{k}^+ -\delta)}(\frac{1}{m}-\frac{1}{2})- \frac{ j(2\sigma- \mathtt{k}^-)- \mathtt{k}^-}{2(\mathtt{k}^+ -\delta)}}\big\||v_t(\tau,\cdot)|^p\big\|_{L^m \cap L^2}d\tau\\
&\qquad + \int_{t/2}^t (1+t-\tau)^{- \frac{j(2\sigma- \mathtt{k}^-)- \mathtt{k}^-}{2(\mathtt{k}^+ -\delta)}}\big\||v_t(\tau,\cdot)|^p\big\|_{L^2}d\tau.
\end{align*}
Furthermore, we can estimate
$$\big\||v_t(\tau,\cdot)|^p\big\|_{L^m \cap L^2} \lesssim \|v_t(\tau,\cdot)\|^p_{L^{mp}}+ \|v_t(\tau,\cdot)\|^p_{L^{2p}},\, \text{ and }\big\||v_t(\tau,\cdot)|^p\big\|_{L^2}= \|v_t(\tau,\cdot)\|^p_{L^{2p}}.$$
Employing the fractional Gagliardo-Nirenberg inequality from Proposition \ref{fractionalGagliardoNirenberg} gives
\begin{align*}
\big\||v_t(\tau,\cdot)|^p\big\|_{L^m \cap L^2} &\lesssim (1+\tau)^{-\frac{n}{2m(\mathtt{k}^+ -\delta)}(p-1)- \frac{p(\sigma- \mathtt{k}^-)}{\mathtt{k}^+ -\delta}}\|(u,v)\|^p_{X(\tau)},\\
\big\||v_t(\tau,\cdot)|^p\big\|_{L^2} &\lesssim (1+\tau)^{-\frac{np}{2(\mathtt{k}^+ -\delta)}(\frac{1}{m}-\frac{1}{2p})- \frac{p(\sigma- \mathtt{k}^-)}{\mathtt{k}^+ -\delta}}\|(u,v)\|^p_{X(\tau)},
\end{align*}
where $p \in \big[\frac{2}{m},\ity \big)$ holds because $s > \frac{n}{2}+ \mathtt{k}^+$. As a result, we get
\begin{align*}
\big\|\partial_t^j u^{nl}(t,\cdot)\big\|_{L^2} &\lesssim \|(u,v)\|^p_{X(t)}(1+t)^{-\frac{n}{2(\mathtt{k}^+ -\delta)}(\frac{1}{m}-\frac{1}{2})- \frac{ j(2\sigma- \mathtt{k}^-)- \mathtt{k}^-}{2(\mathtt{k}^+ -\delta)}} \int_0^{t/2}(1+\tau)^{-\frac{n}{2m(\mathtt{k}^+ -\delta)}(p-1)- \frac{p(\sigma- \mathtt{k}^-)}{\mathtt{k}^+ -\delta}}d\tau\\
&\qquad + \|(u,v)\|^p_{X(t)}(1+t)^{-\frac{np}{2(\mathtt{k}^+ -\delta)}(\frac{1}{m}-\frac{1}{2p})- \frac{p(\sigma- \mathtt{k}^-)}{\mathtt{k}^+ -\delta}} \int_{t/2}^t (1+t-\tau)^{- \frac{j(2\sigma- \mathtt{k}^-)- \mathtt{k}^-}{2(\mathtt{k}^+ -\delta)}}d\tau.
\end{align*}
Here we used $(1+t-\tau) \approx (1+t) \text{ for any }\tau \in [0,t/2] \text{ and } (1+\tau) \approx (1+t) \text{ for any }\tau \in [t/2,t] $. Since the condition $p > 1+\frac{2m\delta}{n}$ holds, the term $(1+\tau)^{-\frac{n}{2m(\mathtt{k}^+ -\delta)}(p-1)- \frac{p(\sigma- \mathtt{k}^-)}{\mathtt{k}^+ -\delta}\big)}$ is integrable. Moreover, we also derive
\begin{align*}
&(1+t)^{-\frac{np}{2(\mathtt{k}^+ -\delta)}(\frac{1}{m}-\frac{1}{2p})- \frac{p(\sigma- \mathtt{k}^-)}{\mathtt{k}^+ -\delta}} \int_{t/2}^t (1+t-\tau)^{- \frac{j(2\sigma- \mathtt{k}^-)- \mathtt{k}^-}{2(\mathtt{k}^+ -\delta)}}d\tau\\ 
&\qquad \lesssim (1+t)^{-\frac{n}{2(\mathtt{k}^+ -\delta)}(\frac{1}{m}-\frac{1}{2})- \frac{j(2\sigma- \mathtt{k}^-)- \mathtt{k}^-}{2(\mathtt{k}^+ -\delta)}},
\end{align*}
due to $\frac{j(2\sigma- \mathtt{k}^-)- \mathtt{k}^-}{2(\mathtt{k}^+ -\delta)}< 1$ for $j=0,1$. Consequently, we may conclude the following estimate for $j=0,1$:
\begin{equation}
\big\|\partial_t^j u^{nl}(t,\cdot)\big\|_{L^2} \lesssim (1+t)^{-\frac{n}{2(\mathtt{k}^+ -\delta)}(\frac{1}{m}-\frac{1}{2})- \frac{ j(2\sigma- \mathtt{k}^-)- \mathtt{k}^-}{2(\mathtt{k}^+ -\delta)}} \|(u,v)\|^p_{X(t)}. \label{t6A1}
\end{equation}
In the second step, let us control the norm $\big\||D|^s u^{nl}(t,\cdot)\big\|_{L^2}$. We have
\begin{align*}
\big\||D|^{s} u^{nl}(t,\cdot)\big\|_{L^2} &\lesssim \int_0^{t/2}(1+t-\tau)^{-\frac{n}{2(\mathtt{k}^+ -\delta)}(\frac{1}{m}-\frac{1}{2})- \frac{s - \mathtt{k}^-}{2(\mathtt{k}^+ -\delta)}}\big\||v_t(\tau,\cdot)|^p\big\|_{L^m \cap L^2 \cap \dot{H}^{s- \mathtt{k}^+}}d\tau\\
&\qquad + \int_{t/2}^t (1+t-\tau)^{- \frac{s - \mathtt{k}^-}{2(\mathtt{k}^+ -\delta)}}\big\||v_t(\tau,\cdot)|^p\big\|_{L^2 \cap \dot{H}^{s- \mathtt{k}^+}}d\tau.
\end{align*}
The integrals with $\big\||v_t(\tau,\cdot)|^p\big\|_{L^m \cap L^2}$ and $\big\||v_t(\tau,\cdot)|^p\big\|_{L^ 2}$ will be handled as before to derive (\ref{t6A1}) provided that the condition (\ref{exponent5A2}) is satisfied. To control the integral with $\big\||v_t(\tau,\cdot)|^p\big\|_{\dot{H}^{s- \mathtt{k}^+}}$, we shall apply Corollary \ref{Corfractionalhomogeneous} for fractional powers with $s- \mathtt{k}^+ \in \big(\frac{n}{2},p\big)$ and Lemma \ref{LemmaEmbedding} with a suitable $s^* <\frac{n}{2}$. Hence, we get
$$ \big\||v_t(\tau,\cdot)|^p\big\|_{\dot{H}^{s- \mathtt{k}^+}} \lesssim \|v_t(\tau,\cdot)\|_{\dot{H}^{s- \mathtt{k}^+}}\|v_t(\tau,\cdot)\|^{p-1}_{L^\ity} \lesssim \|v_t(\tau,\cdot)\|_{\dot{H}^{s- \mathtt{k}^+}}\big(\|v_t(\tau,\cdot)\|_{\dot{H}^{s^*}}+ \|v_t(\tau,\cdot)\|_{\dot{H}^{s- \mathtt{k}^+}}\big)^{p-1}. $$
Applying the fractional Gagliardo-Nirenberg inequality from Proposition \ref{fractionalGagliardoNirenberg} we have
$$ \|v_t(\tau,\cdot)\|_{\dot{H}^{s^*}} \lesssim \|v_t(\tau,\cdot)\|^{1-\theta}_{L^2}\big\||D|^{s- \mathtt{k}^+} v_t(\tau,\cdot)\big\|^{\theta}_{L^2} \lesssim (1+\tau)^{-\frac{n}{2(\mathtt{k}^+ -\delta)}(\frac{1}{m}-\frac{1}{2})- \frac{\sigma- \mathtt{k}^-}{\mathtt{k}^+ -\delta}-\frac{s^*}{2(\mathtt{k}^+ -\delta)}}\|(u,v)\|^p_{X(\tau)}, $$
where $\theta= \frac{s^*}{s- \mathtt{k}^+}$. Consequently, we obtain
\begin{align*}
&\big\||v_t(\tau,\cdot)|^p\big\|_{\dot{H}^{s- \mathtt{k}^+}} \lesssim (1+\tau)^{-\frac{np}{2(\mathtt{k}^+ -\delta)}(\frac{1}{m}-\frac{1}{2})- \frac{p(\sigma- \mathtt{k}^-)}{\mathtt{k}^+ -\delta}-\frac{s- \mathtt{k}^+ }{2(\mathtt{k}^+ -\delta)}-(p-1)\frac{s^*}{2(\mathtt{k}^+ -\delta)}}\|(u,v)\|^p_{X(\tau)}\\
&\qquad \lesssim (1+\tau)^{-\frac{np}{2(\mathtt{k}^+ -\delta)}(\frac{1}{m}-\frac{1}{2p})- \frac{p(\sigma- \mathtt{k}^-)}{\mathtt{k}^+ -\delta}}\|(u,v)\|^p_{X(\tau)},
\end{align*}
if we choose $s^*= \frac{n}{2}- \e$, where $\e$ is a sufficiently small positive. Similar to the above arguments we arrive at
\begin{align}
\big\||D|^{s} u^{nl}(t,\cdot)\big\|_{L^2} &\lesssim (1+t)^{-\frac{n}{2(\mathtt{k}^+ -\delta)}(\frac{1}{m}-\frac{1}{2})- \frac{s - \mathtt{k}^-}{2(\mathtt{k}^+ -\delta)}}\|(u,v)\|^p_{X(t)}, \label{t6A2} \\ 
\big\||D|^{s- \mathtt{k}^+}u_t^{nl}(t,\cdot)\big\|_{L^2} &\lesssim (1+t)^{-\frac{n}{2(\mathtt{k}^+ -\delta)}(\frac{1}{m}-\frac{1}{2})- \frac{s+ \mathtt{k}^+ - 4\delta}{2(\mathtt{k}^+ -\delta)}}\|(u,v)\|^p_{X(t)}. \label{t6A3}
\end{align}
In the analogous way, we also derive for $j=0,1$ the following estimates:
\begin{align}
\big\|\partial_t^j v^{nl}(t,\cdot)\big\|_{L^2} &\lesssim (1+t)^{-\frac{n}{2(\mathtt{k}^+ -\delta)}(\frac{1}{m}-\frac{1}{2})- \frac{j(2\sigma- \mathtt{k}^-)}{2(\mathtt{k}^+ -\delta)}}\|(u,v)\|^q_{X(t)}, \label{t6A4} \\
\big\||D|^{s}v^{nl}(t,\cdot)\big\|_{L^2} &\lesssim (1+t)^{-\frac{n}{2(\mathtt{k}^+ -\delta)}(\frac{1}{m}-\frac{1}{2})- \frac{s- \mathtt{k}^-}{2(\mathtt{k}^+ -\delta)}}\|(u,v)\|^q_{X(t)}, \label{t6A5} \\
\big\||D|^{s- \mathtt{k}^+}v_t^{nl}(t,\cdot)\big\|_{L^2} &\lesssim (1+t)^{-\frac{n}{2(\mathtt{k}^+ -\delta)}(\frac{1}{m}-\frac{1}{2})- \frac{s+ \mathtt{k}^+ - 4\delta}{2(\mathtt{k}^+ -\delta)}}\|(u,v)\|^q_{X(t)}, \label{t6A6}
\end{align}
where the condition (\ref{exponent5A2}) holds. From (\ref{t6A1}) to (\ref{t6A6}) and the definition of the norm in $X(t)$ we may conclude immediately the inequality (\ref{pt4.31}).\medskip

Next, let us prove the inequality (\ref{pt4.4}). The difficulty appearing is to deal with estimating the following terms:
$$\big\||v_t(\tau,\cdot)|^p-|\bar{v}_t(\tau,\cdot)|^p\big\|_{\dot{H}^{s- \mathtt{k}^+}},\,\,\big\||u_t(\tau,\cdot)|^q-|\bar{u}_t(\tau,\cdot)|^q\big\|_{\dot{H}^{s- \mathtt{k}^+}}. $$
Then, repeating the proof of Theorem $3$ and using the analogous treatment as in the above steps we may conclude the inequality (\ref{pt4.4}). Summarizing, the proof of Theorem $5$ is completed.

\subsection{Proof of Theorem $6$: $s_1 \ge s_2 > \frac{n}{2}+ \mathtt{k}^+$}  
We follow ideas from Theorems $4$ and $5$. We introduce the solution space
$$X(t):= \Big(C([0,\ity),H^{s_1})\cap C^1([0,\ity),H^{s_1- \mathtt{k}^+})\Big) \times \Big(C([0,\ity),H^{s_2})\cap C^1([0,\ity),H^{s_2- \mathtt{k}^+})\Big). $$
Then, repeating some steps of the proofs we did in Theorems $4$ and $5$ we may complete the proof of Theorem $6$.

\section{Optimality of the exponents} \label{Optimality}
In this section, our goal is to find really a critical exponent from theorems in main results. First, let us consider the following Cauchy problem:
\begin{equation} \label{pt4.1}
u_{tt}+ (-\Delta)^\sigma u+ (-\Delta)^{\delta} u_t=|u|^p,\,\,\, u(0,x)=0,\,\,\, u_t(0,x)=u_1(x),
\end{equation}
with some $\sigma \ge 1$, $\delta\in (0,\sigma)$ and a given real number $p>1$. Here, critical exponent $p_{crit}$ means that for some range of admissible $p>p_{crit}$ there exists a global (in time) solution for small initial data in a suitable space. Moreover, one may find suitable small data such that there exists no global (in time) solution in the case $1< p \le p_{crit}$. In other word, we have only local (in time) solution.  \medskip

Now let us consider the Cauchy problem for the following system: 
\begin{equation}
\begin{cases}
u_{tt}+ (-\Delta)^\sigma u+ (-\Delta)^{\delta} u_t=|v|^p,\,\,\,  v_{tt}+ (-\Delta)^\sigma v+ (-\Delta)^{\delta} v_t=|u|^q, \\
u(0,x)= u_0(x),\,\, u_t(0,x)=u_1(x),\,\, v(0,x)= v_0(x),\,\, v_t(0,x)=v_1(x), \label{pt4.2}
\end{cases}
\end{equation}
with $\sigma \ge 1$, $\delta \in (0,\sigma)$ and $p,\, q >1$. Recently, there are several papers concerning some special cases of $\sigma$ and $\delta$ including $\sigma=1$ and $\delta=0$ in \cite{NishiharaWakasugi,SunWang}, or $\sigma=1$ and $\delta=\frac{1}{2}$ in \cite{Dabbicco}. To state our result, we recall the following definition of weak solution to (\ref{pt4.2}) (see, for instance, \cite{DabbiccoEbert,SunWang}).
\begin{dn} \label{defweaksolution1}
Let $p,\, q>1$ and $T>0$. We say that $(u,v) \in L^q_{loc}([0,T)\times \R^n) \times L^p_{loc}([0,T)\times \R^n)$ is a weak solution to (\ref{pt4.2}) if for any test function $\phi(t,x) \in \mathcal{C}_0^\ity([0,T)\times \R^n)$ it holds:
\begin{align}
&\int_0^T \int_{\R^n}|v(t,x)|^p \phi(t,x)dxdt+ \int_{\R^n}\big(u_1(x)+ (-\Delta)^{\delta}u_0(x)\big)\phi(0,x)dx \nonumber \\
&\qquad = \int_{\R^n}u_0(x)\phi_t(0,x)+ \int_0^T \int_{\R^n}u(t,x)\big(\phi_{tt}(t,x)- (-\Delta)^{\delta}\phi_t(t,x)+ (-\Delta)^{\sigma}\phi(t,x) \big)dxdt \label{ptweaksolution1}
\end{align}
and
\begin{align}
&\int_0^T \int_{\R^n}|u(t,x)|^q \phi(t,x)dxdt+ \int_{\R^n}\big(v_1(x)(-\Delta)^{\delta}v_0(x)\big)\phi(0,x)dx \nonumber \\
&\qquad = \int_{\R^n}v_0(x)\phi_t(0,x)+ \int_0^T \int_{\R^n}v(t,x)\big(\phi_{tt}(t,x)- (-\Delta)^{\delta}\phi_t(t,x)+ (-\Delta)^{\sigma}\phi(t,x) \big)dxdt. \label{ptweaksolution2}
\end{align}
\end{dn}
If $T= \ity$, we say that $(u,v)$ is a global weak solution to (\ref{pt4.2}).

The proof of blow-up results in this section is based on a contradiction argument by using the test function method (see, for example, \cite{Wakasugi,Zhang}). In general, this method cannot be directly applied to fractional Laplacian operators $(-\Delta)^\sigma$ and $(-\Delta)^\delta$ as well-known non-local operators. Hence, the assumption for integers $\sigma$ and $\delta$ comes into play in our proof. Moreover, the test function method is not influenced by higher regularity of the data. For this reason, we restrict ourselves to the sharpness of the critical exponent to (\ref{pt4.2}) where the data are supposed to belong to the energy space, i.e., as in Theorem 1-A. The ideas of the proof of the following theorem are based on the paper \cite{DabbiccoEbert} focusing on studying (\ref{pt4.1}). We shall prove the following result.
\begin{dl} \label{dloptimal10.2.1}
Let $\sigma,\,\delta\in \N \setminus \{0\}$ and $\delta \in (0,\sigma)$. We assume that the initial data $\big((u_0,u_1),\, (v_0,v_1) \big) \in \mathcal{A}^{\mathtt{k}^+}_{1} \times \mathcal{A}^{\mathtt{k}^+}_{1}$ satisfies the following relations:
\begin{equation} \label{optimal1}
\liminf_{R\longrightarrow \ity} \int_{\R^n} \big(u_1(x)+ (-\Delta)^\delta u_0(x)\big)dx >0,\qquad \quad \liminf_{R\longrightarrow \ity} \int_{\R^n} \big(v_1(x)+ (-\Delta)^\delta v_0(x)\big)dx >0.
\end{equation}
Moreover, we suppose the condition
\begin{equation} \label{optimal2}
n \le \mathtt{k}^- + \frac{2\sigma(1+\max\{p,\,q\})}{pq-1}.
\end{equation}
Then, there is no global (in time) energy solution  to (\ref{pt4.2}) and the blow-up time $T_\e$ is estimated by
\begin{equation} \label{lifetime}
T_\e \le C\e^{-\frac{2\sigma- \mathtt{k}^-}{\mathtt{k}^- + \frac{2\sigma(q+1)}{pq}- n}} \text{ with }C>0 \text{ and a small constant } \e.
\end{equation}
\end{dl}

\begin{proof}
First, we introduce test functions $\eta= \eta(t)$ and $\varphi=\varphi(x)$ having the following properties:
\begin{align}
&1.\quad \eta \in \mathcal{C}_0^\ity([0,\ity)) \text{ and }
\eta(t)=\begin{cases}
1 \text{ for }0 \le t \le 1/2, \\
0 \text{ for }t \ge 1,
\end{cases} & \nonumber \\ 
&2.\quad \varphi \in \mathcal{C}_0^\ity(\R^n) \text{ and }
\varphi(x)= \begin{cases}
1 \text{ for } |x|\le 1/2, \\
0 \text{ for }|x|\ge 1,
\end{cases} & \nonumber \\
&3.\quad \eta^{-\frac{\kappa'}{\kappa}}\big(|\eta'|^{\kappa'}+|\eta''|^{\kappa'}\big) \text{ and } \varphi^{-\frac{\kappa'}{\kappa}}\big(|\Delta^{\delta}\varphi|^{\kappa'}+|\Delta^{\sigma}\varphi|^{\kappa'}\big) \text{ are bounded, } & \label{condition3}
\end{align}
with $\kappa= p,\,q$, where $\kappa'$ is the conjugate of $\kappa$. Moreover, we assume that $\eta(t)$ is a decreasing function and that $\varphi=\varphi(|x|)$ is a radial function with $\varphi(|x|) \le \varphi(|y|)$ for any $x,y$ such that $|x|\ge |y|$.\medskip

\noindent Let $R$ be a large parameter in $[0,\ity)$. We define the following test function:
$$ \phi_R(t,x):= \eta_R(t) \varphi_R(x), $$
where $\eta_R(t):= \eta(R^{-\alpha}t)$ and $\varphi_R(x):= \varphi(R^{-1}x)$ for a fixed constant $\alpha:= 2\sigma- \mathtt{k}^-$. We define the funtional
$$ I_R:= \int_0^{\ity}\int_{\R^n}|v(t,x)|^p \phi_R(t,x) dxdt= \int_{Q_R}|v(t,x)|^p \phi_R(t,x) dxdt, $$
and
$$ J_R:= \int_0^{\ity}\int_{\R^n}|u(t,x)|^q \phi_R(t,x) dxdt= \int_{Q_R}|u(t,x)|^q \phi_R(t,x) dxdt, $$
where $$Q_R:= [0,R^{\alpha}] \times B_R,\,\, B_R:= \big\{x\in \R^n: |x|\le R \big\}. $$
Let us assume that $(u,v)= \big(u(t,x),v(t,x)\big)$ is the global solution to (\ref{pt4.2}). By carrying out partial integration, we plug $\phi(t,x)= \phi_R(t,x)$ into (\ref{ptweaksolution1}) to derive
\begin{align} 
&I_R+ \int_{B_R}\big(u_1(x)+ (-\Delta)^{\delta}u_0(x)\big)\varphi_R(x)dx \nonumber \\
&\quad = \int_{Q_R}u(t,x) \Big(\eta''_R(t) \varphi_R(x)- \eta'_R(t) (-\Delta)^{\delta}\varphi_R(x)+ \eta_R(t) (-\Delta)^{\sigma}\varphi_R(x)\Big)dxdt \label{t4.1.1}
\end{align}
Applying H\"{o}lder's inequality with $\frac{1}{q}+\frac{1}{q'}=1$ we can estimate as follows:
\begin{align*}
&\int_{Q_R} |u(t,x)|\, \big|\eta'_R(t) \varphi_R(x)\big| dxdt \le \Big(\int_{Q_R} \Big|u(t,x)\phi^{\frac{1}{q}}_R(t,x)\Big|^q dxdt\Big)^{\frac{1}{q}} \Big(\int_{Q_R} \Big|\phi^{-\frac{1}{q}}_R(t,x) \eta''_R(t) \varphi_R(x)\Big|^{q'} dxdt\Big)^{\frac{1}{q'}}, \\
&\qquad \qquad \le J_R^{\frac{1}{q}} \Big( \int_{Q_R}\eta_R^{-\frac{q'}{q}}(t) \big|\eta''_R(t)\big|^{q'} \varphi_R(x) dxdt\Big)^{\frac{1}{q'}}.
\end{align*}
By change of variables $\tilde{t}:= R^{-\alpha}t$ and $\tilde{x}:= R^{-1}x$, we get
\begin{equation} \label{t4.1.2}
\int_{Q_R} |u(t,x)|\, \big|\eta'_R(t) \varphi_R(x)\big| dxdt \lesssim J_R^{\frac{1}{q}}\, R^{-2\alpha+ \frac{n+\alpha}{q'}}, 
\end{equation}
Here we used $ \eta''_R(t)= R^{-2\alpha}\eta''(\tilde{t})$ and the assumption (\ref{condition3}). In the same way, we also can estimate
\begin{align}
&\int_{Q_R}|u(t,x)|\, \big|\eta'_R(t) (-\Delta)^{\delta}\varphi_R(x)\big|dxdt \lesssim J_R^{\frac{1}{q}}\, R^{-\alpha-2\delta+ \frac{n+\alpha}{q'}}, \label{t4.1.3} \\ 
&\int_{Q_R}|u(t,x)|\, \big|\eta_R(t) (-\Delta)^{\sigma}\varphi_R(x)\big|dxdt \lesssim J_R^{\frac{1}{q}}\, R^{-2\sigma+ \frac{n+\alpha}{q'}}, \label{t4.1.4}
\end{align}
where we note that
$$ \eta'_R(t)= R^{-\alpha}\eta'(\tilde{t}),\,\, (-\Delta)^{\delta}\varphi_R(x)= R^{-2\delta}(-\Delta)^{\delta}\varphi(\tilde{x}) \text{ and } (-\Delta)^{\sigma}\varphi_R(x)= R^{-2\sigma}(-\Delta)^{\sigma}\varphi(\tilde{x}). $$
Because of the assumption (\ref{optimal1}), there exists a constant $R_0>0$ such that it holds for any $R>R_0$
\begin{equation} \label{t4.1.5}
\int_{B_R}\big(u_1(x)+ (-\Delta)^{\delta}u_0(x)\big)\varphi_R(x)dx >0.
\end{equation}
Consequently, from (\ref{t4.1.1}) to (\ref{t4.1.5}) we may conclude the following estimate:
\begin{equation} \label{t4.1.6}
I_R \lesssim J_R^{\frac{1}{q}}\, R^{-2\sigma+ \frac{n+\alpha}{q'}}.
\end{equation}
Analogously, we also arrive at
\begin{equation} \label{t4.1.7}
J_R \lesssim I_R^{\frac{1}{p}}\, R^{-2\sigma+ \frac{n+\alpha}{p'}}.
\end{equation}
From (\ref{t4.1.6}) and (\ref{t4.1.7}) we obtain
\begin{align}
I_R^{\frac{pq-1}{pq}} &\lesssim R^{(-2\sigma+ \frac{n+\alpha}{p'})\frac{1}{q}- 2\sigma+ \frac{n+\alpha}{q'}}=: R^{\beta_1}, \label{t4.1.8}\\ 
J_R^{\frac{pq-1}{pq}} &\lesssim R^{(-2\sigma+ \frac{n+\alpha}{q'})\frac{1}{p}- 2\sigma+ \frac{n+\alpha}{p'}}=: R^{\beta_2} \label{t4.1.9}.
\end{align}
Without loss of generality we can assume $q>p$. The assumption (\ref{optimal2}) becomes
$$n \le \mathtt{k}^- + \frac{2\sigma(1+q)}{pq-1}, $$
that is, $\beta_2 \le 0$. We shall split our consideration into two cases. In the first case $\beta_2 <0$, letting $R\longrightarrow \ity$ in (\ref{t4.1.9}) we have
$$ \int_0^{\ity}\int_{\R^n}|u(t,x)|^q dxdt= 0, $$
which follows $u \equiv 0$. This is a contradiction to the assumption (\ref{optimal1}). In the second case $\beta_2=0$, from (\ref{t4.1.9}) there exists a positive constant $C$ such that
$$ \int_0^{\ity}\int_{\R^n}|u(t,x)|^q \phi_R(t,x) dxdt \le C, $$
for a sufficiently large $R$. This implies
\begin{equation} \label{t4.1.10}
\int_{\tilde{Q}_R}|u(t,x)|^q \phi_R(t,x) dxdt \longrightarrow 0 \text{ as } R\longrightarrow \ity,
\end{equation}
where we introduce notation
$$\tilde{Q}_R:= Q_R \setminus \big([0,R^{\alpha}/2] \times B_{R/2}\big),\,\, B_{R/2}:= \big\{x\in \R^n: 0\le |x|\le R/2 \big\}. $$
Due to $\partial^2_t \phi_R(t,x)= (-\Delta)^{\delta}\partial_t\phi_R(t,x)= (-\Delta)^{\delta}\phi_R(t,x)=0$ in $(\R^1_+ \times \R^n) \setminus \tilde{Q}_R$, repeating the steps of the proof from (\ref{t4.1.1}) to (\ref{t4.1.5}) we may conclude the following estimates:
\begin{align*}
&I_R+ \int_{B_R}\big(u_1(x)+ (-\Delta)^{\delta}u_0(x)\big)\varphi_R(x)dx \lesssim \Big(\int_{\tilde{Q}_R}|u(t,x)|^q \phi_R(t,x) dxdt\Big)^{\frac{1}{q}}\, R^{-2\sigma+ \frac{n+\alpha}{q'}}, \\ 
&J_R+ \int_{B_R}\big(v_1(x)+ (-\Delta)^{\delta}v_0(x)\big)\varphi_R(x)dx \lesssim \Big(\int_{\tilde{Q}_R}|v(t,x)|^p \phi_R(t,x) dxdt\Big)^{\frac{1}{p}}\, R^{-2\sigma+ \frac{n+\alpha}{p'}}.
\end{align*}
Because $\beta_2= 0$, from both the above estimates and (\ref{t4.1.5}) we get
\begin{equation} \label{t4.1.11}
J_R+ \int_{B_R}\big(v_1(x)+ (-\Delta)^{\delta}v_0(x)\big)\varphi_R(x)dx \lesssim \Big(\int_{\tilde{Q}_R}|u(t,x)|^q \phi_R(t,x) dxdt\Big)^{\frac{1}{pq}}.
\end{equation}
From (\ref{t4.1.10}) and (\ref{t4.1.11}), letting $R\longrightarrow \ity$ we arrive at
$$ \int_0^{\ity}\int_{\R^n}|u(t,x)|^q dxdt+ \int_{\R^n}\big(v_1(x)+ (-\Delta)^{\delta}v_0(x)\big)dx= 0, $$
which is again a contradiction to the assumption (\ref{optimal1}).\medskip

Let us now consider the case of subcritical exponents. We assume that $(u,v)= \big(u(t,x),v(t,x)\big)$ is the local solution in $([0,T)\times \R^n) \times ([0,T)\times \R^n)$ to (\ref{pt4.2}). In order to prove the lifespan estimate, we replace the initial data $\big((u_0,u_1),\, (v_0,v_1) \big)$ by $\big((\e f_0,\e f_1),\, (\e g_0,\e g_1) \big)$ with a small constant $\e$. Here $\big((f_0,f_1),\, (g_0,g_1) \big) \in \mathcal{A}^{\mathtt{k}^+}_{1} \times \mathcal{A}^{\mathtt{k}^+}_{1}$ satisfies the assumption (\ref{optimal1}). Repeating the steps in the above proofs we arrive at the following estimtes:
\begin{align}
&I_R+ c\e \le C\, J_R^{\frac{1}{q}}\, R^{-2\sigma+ \frac{n+\alpha}{q'}}, \label{t4.1.12} \\ 
&J_R+ c\e \le C\, I_R^{\frac{1}{p}}\, R^{-2\sigma+ \frac{n+\alpha}{p'}}. \label{t4.1.13}
\end{align}
Plugging $I_R$ from (\ref{t4.1.12}) into (\ref{t4.1.13}) we derive
\begin{equation}
c\e \le C\, J_R^{\frac{1}{pq}}\, R^{-2\sigma+ \frac{n+\alpha}{p'}+(-2\sigma+ \frac{n+\alpha}{q'})\frac{1}{p}}- J_R. \label{t4.1.14}
\end{equation}
Applying the following elementary inequality
$$ A\,y^\gamma- y \le A^{\frac{1}{1-\gamma}} \text{ for any } A>0,\, y \ge 0 \text{ and } 0< \gamma< 1, $$
to (\ref{t4.1.14}) and a standard calculation lead to
$$ \e \le C\, R^{-\big(\mathtt{k}^- + \frac{2\sigma(q+1)}{pq}- n\big)} \le C\, T^{-\frac{1}{\alpha}\big(\mathtt{k}^- + \frac{2\sigma(q+1)}{pq}- n\big)}, $$
with $R= T^{\frac{1}{\alpha}}$. Finally, letting $T\longrightarrow T_\e - 0$ we may conclude (\ref{lifetime}). Summarizing, the proof of Theorem \ref{dloptimal10.2.1} is completed.
\end{proof}

\begin{nx}
\fontshape{n}
\selectfont
If we replace $m=1$ in Theorems 1-A, then from Theorem \ref{dloptimal10.2.1} it is clear that the exponent given by (\ref{exponent1A}) is really critical in the case $\delta \in (0,\frac{\sigma}{2}]$. However, in the remaining $\delta \in (\frac{\sigma}{2},\sigma)$ there appears a gap between the exponent in (\ref{exponent1A}) and the exponent in (\ref{optimal2}).
\end{nx}

\section{Concluding remarks and open problems} \label{ConcludeOpen}

\begin{nx}{($(L^m \cap L^q)- L^q$ and $L^q- L^q$ estimates)}
\fontshape{n}
\selectfont
In the present paper, we derived $(L^m \cap L^2)- L^2$ and $L^2- L^2$ estimates for solutions and some its derivatives to (\ref{pt1.3}) to prove the global (in time) existence of small data Sobolev solutions to weakly coupled systems of semilinear structurally damped $\sigma$-evolution models with different power nonlinearities (\ref{pt1.1}), (\ref{pt1.2}) and (\ref{pt1.4}). More general, a next challenge is to obtain the global (in time) existence results to (\ref{pt1.1}), (\ref{pt1.2}) and (\ref{pt1.4}) by using $(L^m \cap L^q)- L^q$ and $L^q- L^q$ estimates for solutions and its derivatives to (\ref{pt1.3}) with $q\in (1,\ity)$ and $m\in [1,q)$. In a forthcoming paper, we will study the benefits from the flexibility of suitable parameters $m$ and $q$ in the treatment of weakly coupled systems of the semi-linear models.
\end{nx}

\begin{nx}{(Time-dependent coefficients)}
\fontshape{n}
\selectfont
It can be also expected to study the global (in time) existence of small data Sobolev solutions from suitable spaces to weakly coupled systems of semilinear structurally damped $\sigma$-evolution models with time-dependent coefficients and different power nonlinearities. In some recent papers (see, for instance, \cite{DabbiccoEbert2016,Kainane}), a classification of the damping term was introduced to investigate the following linear Cauchy problem:
$$ u_{tt}+ (-\Delta)^\sigma u+ b(t) (-\Delta)^{\delta} u_t= 0 ,\,\,\, u(0,x)= u_0(x),\,\,\, u_t(0,x)=u_1(x). $$
For this reason, it is interesting to consider the following Cauchy problems:
\begin{equation}
\begin{cases}
u_{tt}+ (-\Delta)^\sigma u+ b(t)(-\Delta)^{\delta} u_t=|v|^p,\,\,\,  v_{tt}+ (-\Delta)^\sigma v+ b(t)(-\Delta)^{\delta} v_t=|u|^q, \\
u(0,x)= u_0(x),\,\, u_t(0,x)=u_1(x),\,\, v(0,x)= v_0(x),\,\, v_t(0,x)=v_1(x), \label{pt5.1}
\end{cases}
\end{equation}
and
\begin{equation}
\begin{cases}
u_{tt}+ (-\Delta)^\sigma u+ b(t)(-\Delta)^{\delta} u_t=|v_t|^p,\,\,\,  v_{tt}+ (-\Delta)^\sigma v+ b(t)(-\Delta)^{\delta} v_t=|u_t|^q, \\
u(0,x)= u_0(x),\,\, u_t(0,x)=u_1(x),\,\, v(0,x)= v_0(x),\,\, v_t(0,x)=v_1(x), \label{pt5.2}
\end{cases}
\end{equation}
and
\begin{equation}
\begin{cases}
u_{tt}+ (-\Delta)^\sigma u+ b(t)(-\Delta)^{\delta} u_t=|v|^p,\,\,\,  v_{tt}+ (-\Delta)^\sigma v+ b(t)(-\Delta)^{\delta} v_t=|u_t|^q, \\
u(0,x)= u_0(x),\,\, u_t(0,x)=u_1(x),\,\, v(0,x)= v_0(x),\,\, v_t(0,x)=v_1(x), \label{pt5.3}
\end{cases}
\end{equation}
with $\sigma \ge 1$, $\delta \in (0,\sigma)$ and $p,\, q >1$. Here the coefficient $b=b(t)$ in (\ref{pt5.1}), (\ref{pt5.2}) and (\ref{pt5.3}) should satisfy some ``effectiveness assumptions'' as in \cite{Kainane}.
\end{nx}





\noindent \textbf{Acknowledgments}\medskip

\noindent The PhD study of MSc. T.A. Dao is supported by Vietnamese Government's Scholarship. The author would like to express sincere thankfulness to Prof. Michael Reissig for his many helpful suggestions and Institute of Applied Analysis for their hospitality. The author wishes to thank the referee for his careful reading of the manuscript and for valuable comments.\medskip

\noindent\textbf{Appendix A}\medskip

\noindent \textit{A.1. Fractional Gagliardo-Nirenberg inequality}

\begin{md} \label{fractionalGagliardoNirenberg}
Let $1<p,\, p_0,\, p_1<\infty$, $\sigma >0$ and $s\in [0,\sigma)$. Then, it holds the following fractional Gagliardo-Nirenberg inequality for all $u\in L^{p_0} \cap \dot{H}^\sigma_{p_1}$:
$$ \|u\|_{\dot{H}^{s}_p}\lesssim \|u\|_{L^{p_0}}^{1-\theta}\,\, \|u\|_{\dot{H}^{\sigma}_{p_1}}^\theta, $$
where $\theta=\theta_{s,\sigma}(p,p_0,p_1)=\frac{\frac{1}{p_0}-\frac{1}{p}+\frac{s}{n}}{\frac{1}{p_0}-\frac{1}{p_1}+\frac{\sigma}{n}}$ and $\frac{s}{\sigma}\leq \theta\leq 1$ .
\end{md}
For the proof one can see \cite{Ozawa}.
\medskip

\noindent \textit{A.2. Fractional Leibniz rule}

\begin{md} \label{fractionalLeibniz}
Let us assume $s>0$ and $1\leq r \leq \infty, 1<p_1,\, p_2,\, q_1,\, q_2 \le \infty$ satisfying the relation \[ \frac{1}{r}=\frac{1}{p_1}+\frac{1}{p_2}=\frac{1}{q_1}+\frac{1}{q_2}.\]
Then, the following fractional Leibniz rule holds:
$$\|\,|D|^s(u \,v)\|_{L^r}\lesssim \|\,|D|^s u\|_{L^{p_1}}\, \|v\|_{L^{p_2}}+\|u\|_{L^{q_1}}\, \|\,|D|^s v\|_{L^{q_2}} $$
for any $u\in \dot{H}^s_{p_1} \cap L^{q_1}$ and $v\in \dot{H}^s_{q_2} \cap L^{p_2}$.
\end{md}
These results can be found in \cite{Grafakos}.
\medskip

\noindent \textit{A.3. Fractional chain rule}

\begin{md} \label{Propfractionalchainrulegeneral}
Let us choose $s>0$, $p>\lceil s \rceil$
 and $1<r,\, r_1,\, r_2<\infty$ satisfying $\frac{1}{r}=\frac{p-1}{r_1}+\frac{1}{r_2}$. Let us denote by $F(u)$ one of the functions $|u|^p, \pm |u|^{p-1}u$. Then, it holds the following fractional chain rule:
$$ \|\,|D|^{s} F(u)\|_{L^r}\lesssim \|u\|_{L^{r_1}}^{p-1}\, \|\,|D|^{s} u\|_{L^{r_2}} $$
for any $u\in  L^{r_1} \cap \dot{H}^{s}_{r_2}$.
 \end{md}
The proof can be found in \cite{Palmierithesis}.
\medskip

\noindent \textit{A.4. Fractional powers}

\begin{md} \label{PropSickelfractional}
Let $p>1$, $1< r <\infty$ and $u \in H^{s}_r$, where $s \in \big(\frac{n}{r},p\big)$.
Let us denote by $F(u)$ one of the functions $|u|^p,\, \pm |u|^{p-1}u$ with $p>1$. Then, the following estimate holds$:$
$$\Vert F(u)\Vert_{H^{s}_r}\lesssim \|u\|_{H^{s}_r}\,\, \|u\|_{L^\infty}^{p-1}.$$
\end{md}

\begin{hq} \label{Corfractionalhomogeneous}
Under the assumptions of Proposition \ref{PropSickelfractional} it holds: $\| F(u)\|_{\dot{H}^{s}_r}\lesssim \|u\|_{\dot{H}^{s}_r}\,\, \|u\|_{L^\infty}^{p-1}.$
\end{hq}
The proof can be found in \cite{DuongKainaneReissig}.
\medskip

\noindent \textit{A.5. Useful lemma}

\bbd \label{LemmaL1normEstimate}
It holds for small frequencies:
\begin{equation}
\int_{\R^n} |\xi|^\beta e^{-c|\xi|^\alpha t}d\xi \lesssim (1+t)^{-\frac{n+\beta}{\alpha}},
\end{equation}
where $n \ge 1$, $\beta \in \R$ satisfying $n+\beta >0$ and for all positive numbers $c,\,\alpha >0$
\ebd
\begin{proof}
We shall split our consideration into two cases. In the first case for $t \in (0,1]$, we get immediately the following estimate:
\begin{equation} \label{L1normEstimate1}
\int_{\R^n} |\xi|^\beta e^{-c|\xi|^\alpha t}d\xi \lesssim \int_0^1 |\xi|^{n+\beta-1} e^{-c|\xi|^\alpha t}d|\xi| \lesssim 1.
\end{equation}
We carry out change of variables $\xi^\alpha t= \eta^\alpha$, that is, $\xi= t^{-\frac{1}{\alpha}} \eta$ in the second step for $t \in [1,\ity)$ to dervie
\begin{equation} \label{L1normEstimate2}
\int_{\R^n} |\xi|^\beta e^{-c|\xi|^\alpha t}d\xi \lesssim t^{-\frac{n+\beta}{\alpha}} \int_0^\ity |\eta|^{n+\beta-1} e^{-c|\eta|^\alpha}d|\eta| \lesssim t^{-\frac{n+\beta}{\alpha}}.
\end{equation}
Hence, from (\ref{L1normEstimate1}) and (\ref{L1normEstimate2}) we can conclude the desired statement.
\end{proof}

\bbd \label{LemmaEmbedding}
Let $0< s_1< \frac{n}{2}< s_2$. Then, for any function $u \in \dot{H}^{s_1} \cap \dot{H}^{s_2}$ we have
\[ \|u\|_{L^\ity} \lesssim \|u\|_{\dot{H}^{s_1}}+ \|u\|_{\dot{H}^{s_2}}. \]
\ebd
The proof can be found in \cite{DabbiccoEbertLucente}.


\end{document}